\documentclass[10pt,reqno]{amsart}
\usepackage[all,cmtip]{xy}
\usepackage{indentfirst}

\usepackage{amssymb}
\usepackage{mathabx}
\usepackage{physics}
\input{macros}
\usepackage[margin = 1.2in]{geometry}

\makeatletter 
\@mparswitchfalse%
\makeatother
\normalmarginpar 

\title{Nonlinear Modal Reduction for Subwavelength Dielectric Scattering}

\author{Habib Ammari} %
\address[H. Ammari]{Department of Mathematics, ETH Z\"{u}rich, R\"{a}mistrasse 101, CH-8092 Z\"{u}rich, Switzerland}
\email{habib.ammari@math.ethz.ch}

\author{Ioan Gorea} %
\address[I. Gorea]{Department of Mathematics, ETH Z\"{u}rich, R\"{a}mistrasse 101, CH-8092 Z\"{u}rich, Switzerland}
\email{igorea@student.ethz.ch}

\author{Bowen Li} %
\address[B. Li]{Department of Mathematics, City University of Hong Kong, Kowloon Tong, Hong Kong SAR}
\email{bowen.li@cityu.edu.hk}

\begin{document}

\begin{abstract}
We study three-dimensional wave scattering by high-index dielectric resonators with Kerr-type nonlinearity under plane-wave incidence, together with the associated nonlinear dielectric scattering resonances, through a nonlinear Lippmann--Schwinger equation. Using a Lyapunov--Schmidt reduction near a simple eigenmode of the Newtonian potential, we obtain a local decomposition of the scattered wave into a resonant contribution and a controlled remainder and derive an explicit nonlinear equation for the resonant modal coefficient for sufficiently small contrast parameter and locally small resonant amplitude and incident field. A second reduction covers the regime of incident fields of order one and weak nonlinearity.
 Our results extend for the first time the linear modal decomposition to nonlinear wave-scattering problems. We complement them by developing a numerical framework based on Nystr\"om discretization, real Newton iteration, and pseudo-arclength continuation. For single resonators of several geometries, our computations confirm the predicted high-contrast resonance scaling and the nonlinear modal approximation, and exhibit the multivalued incident-wave response. For a mirror-symmetric resonator dimer, we track symmetric, antisymmetric, and symmetry-broken resonance branches over a broad range of separations. Our computations show that the symmetry-breaking threshold increases as the gap between the resonators decreases, in agreement with the leading-order bifurcation theory.
\end{abstract}

\maketitle

\tableofcontents

\section{Introduction}
\label{sec:introduction}

Recent developments in nanophotonics have led to increased interest in subwavelength photonic resonators made from all-dielectric materials with high refractive indices \cite{kuznetsov_optically_2016,arbabi_dielectric_2015}. 
Despite advances in the experimental and numerical modeling of nonlinear optical scattering, a deep mathematical understanding of nonlinear subwavelength dielectric resonances remains largely undeveloped. The linear case is better understood, from the classical theory of Mie scattering by spherical nanoparticles \cite{original_mie_spherical_2018} to the Helmholtz problem for arbitrarily shaped nanoparticles treated in \cite{ammari_linear_helmholtz_2019}. A comprehensive mathematical framework has also been developed for the linear dielectric resonance problem for the full Maxwell system, both for single dielectric nanoparticles and for clusters of nanoparticles \cite{ammari_maxwell_2023, sini_maxwell_linear_2022, ammari_fanoresonance_2024}.

Although the mathematical foundations of linear subwavelength physics are now well developed \cite{ammari_mathematical_2026}, nonlinear subwavelength dielectric scattering resonances remain poorly understood, in part because the problem lacks a variational framework. Recent works \cite{meklachi_asymptotic_2018,kosche_nonlinear_2025} address high-contrast nonlinear scatterers and nonlinear acoustic bubbly media, respectively, using asymptotic analysis and assuming a priori the existence of nonlinear resonances with convergent series expansions. Closely related to the present work, the existence of nonlinear subwavelength dielectric resonances bifurcating from the zero solution at the corresponding linear resonances was proved in \cite{ammari_dielectric_2025}. There, the resonance problem associated with the nonlinear Helmholtz equation is characterized by a nonlinear Lippmann--Schwinger equation. The corresponding nonlinear resonant states are analogous to stationary states in nonlinear quantum systems; results for the nonlinear Schr\"odinger equation \cite{kirr_symmetry-breaking_2008,kirr_symmetry-breaking_2011} can therefore be adapted to this scattering problem to establish symmetry-breaking phenomena. We also refer the reader to \cite{clemens2026,yu2026} for closely related recent results.

The aim of this work is to build on these results and further develop a mathematical and numerical framework for nonlinear dielectric resonances in the subwavelength regime, focusing on the nonlinear Helmholtz equation studied in \cite{ammari_dielectric_2025}. Our main contribution is twofold: (i) we determine to what extent linear modal decompositions in terms of subwavelength resonances, such as those derived in \cite{feppon_modal_2022,millien2022}, extend to nonlinear wave-scattering problems; and (ii) we find and track bifurcating branches arising in the nonlinear scattering resonance problem.

To give a quantitative nonlinear modal approximation of the scattered field, we 
start from a nonlinear Lippmann--Schwinger formulation of the scattering problem. Then, we perform a Lyapunov--Schmidt reduction near a simple eigenmode of the Newtonian potential. For a sufficiently small contrast parameter and locally small resonant amplitude and incident field, we prove that the interior field admits a unique decomposition into a resonant mode and a controlled complementary component. This yields an exact scalar equation for the resonant modal amplitude and, at leading order, a cubic input--output relation that predicts multivalued nonlinear response. We also obtain an exterior modal approximation for the scattered field with an explicit remainder estimate. A second reduction covers the regime of incident fields of order one and weak nonlinearity.

To complement our analysis, we develop a Nystr\"om discretization of the volume integral equation together with a real Newton solver. For nonlinear resonance computations, the $S^1$ phase invariance is removed by combining a normalization constraint with a local phase condition. Pseudo-arclength continuation is used to find and track multiple solution branches. The method applies to nonspherical resonators and to multiparticle configurations.

The numerical experiments address three questions. First, for single resonators of several geometries, we verify the high-contrast asymptotics of the nonlinear resonant frequency and study grid refinement. Second, for the incident-wave problem, we compare the full numerical solution with the reduced amplitude equation and the exterior modal approximation; the computations display the predicted resonant line shape and a multivalued nonlinear response. Third, for a mirror-symmetric dimer, we track symmetric, antisymmetric, and asymmetric resonance branches. The observed symmetry-breaking threshold agrees closely with the leading-order bifurcation theory and increases as the gap between the two resonators decreases, indicating that stronger near-field hybridization raises the nonlinear amplitude required for symmetry breaking.

The paper is organized as follows. In Section \ref{sec:theory-reduction}, we introduce the nonlinear scattering problem for high-index dielectric resonators subject to an incident plane wave. Our main results give a local decomposition of the scattered field into a resonant contribution and a controlled remainder, together with an explicit nonlinear equation determining the modal coefficient. Section \ref{sec:numerical} considers a single resonator and a dimer. For the single resonator, we present numerical results for the nonlinear scattering problem introduced in this paper; for the dimer, we study the nonlinear dielectric resonance problem without an incident field, as in \cite{ammari_dielectric_2025}. We explain how to find and track the bifurcating branches predicted by the local theory and present new numerical findings on their dependence on the separation distance between the resonators. Finally, concluding remarks and future directions are given in Section \ref{sec:conclusion}.

Throughout this paper, we use the standard Sobolev spaces on a domain $D \subset \R^3$. Our complex $L^2(D)$ inner product is
$\innerproduct{u}{v}_D:=\int_D\overline{u(x)}v(x)\,\dd x$,
so it is conjugate-linear in the first argument. We also use $\partial$ to denote the Fr\'echet derivative. The symbol $C>0$ denotes a generic positive constant that may vary from line to line; any relevant parameter dependence will be stated. For two normed vector spaces $V$ and $W$, we write $V \hookrightarrow W$ for a continuous embedding and $V \Subset W$ for a compact embedding, as in \cite{ciarlet_fa_2013}.

\section{Nonlinear Dielectric Resonances}\label{sec:theory-reduction}

\subsection{Preliminaries}

Let $D \subset \R^3$ be a bounded open set with boundary of class $C^{1,1}$, and let $\chi_D$ denote its characteristic function. Consider a prescribed incident field $u^{\mathrm{inc}} \in H^2_{\mathrm{loc}}(\R^3)$ satisfying the free-space Helmholtz equation
\begin{equation}\label{eq:freespace-helmholtz-incident-field}
    - (\Delta + \omega^2) u^{\mathrm{inc}} = 0 \quad \mathrm{in} \ \R^3.
\end{equation}
We seek a total field $u \in H^2_{\mathrm{loc}}(\R^3)$, decomposed as $u = u^{\mathrm{inc}} + u^{\mathrm{sc}}$, that satisfies the nonlinear Helmholtz equation
\begin{equation}\label{eq:nonlinear-helmholtz}
    \begin{aligned}
        -\Delta u - \omega^2 u &=  \omega^2 \tau \chi_D q_{\eta}(u)u \quad \mathrm{in} \ \R^3,\\
        u\restrict{+} &= u\restrict{-} \quad \mathrm{on} \ \partial D, \\
        \frac{\partial u}{\partial \nu}\restrict{+} &= \frac{\partial u}{\partial \nu}\restrict{-} \quad \mathrm{on} \ \partial D,\\
    \end{aligned}
\end{equation}
together with the Sommerfeld radiation condition for $\omega > 0$:
\begin{equation}\label{eq:sommerfeld}
    \begin{aligned}
        r \abs{\frac{\partial u^{\mathrm{sc}}}{\partial r} - \iu \omega u^{\mathrm{sc}}} \rightarrow 0, \ \mathrm{as} \ r = \abs{x} \rightarrow \infty,
    \end{aligned}
\end{equation}
where $u^{\mathrm{sc}}$ denotes the scattered field. The function
$q_{\eta}:\C\rightarrow\R$, $q_{\eta}(z):=1+\eta\abs{z}^2$, models the
nonlinear response of the refractive index to the intensity of the wave field,
whereas $\eta > 0$ is a real parameter characterizing the strength of the
nonlinearity.
When $u^{\mathrm{inc}}=0$, recall from \cite{ammari_dielectric_2025} that a nontrivial solution $u \neq 0$ of the eigenvalue problem \eqref{eq:nonlinear-helmholtz} is called a resonant state. This leads to the following definition.

\begin{definition}\label{def:nonlinear-dielec-resonance}
    A complex number $\omega \in \C$ is a nonlinear dielectric scattering resonance for the nonlinear Helmholtz equation without an incident field, i.e., \eqref{eq:nonlinear-helmholtz} with $u^\mathrm{inc} = 0$, if there exists a nontrivial solution $0 \not\equiv u\in H^2(D)$ such that
    \begin{equation}
        u = \tau\omega^2 \K_D^\omega[q_{\eta}(u) u] \quad \text{ in } D.
    \end{equation}
    The nonlinear resonance is called subwavelength if $\omega(\tau) \rightarrow 0$ as $\tau \rightarrow \infty$.
\end{definition}
Recall that, for complex $\omega$, the outgoing radiation condition can be specified using the following free-space Green function of $-(\Delta + \omega^2)$ in $\R^3$:
\begin{equation} \label{Gomega}
    G^\omega(x) := \frac{e^{\iu \omega \abs{x}}}{4 \pi \abs{x}}. 
\end{equation}
In fact, it can be understood by analytic continuation of the free outgoing resolvent, or equivalently, through the kernel $G^\omega$; see, for instance, \cite{zworski}. 

The associated volume potential operator, which is the restriction to $L^2(D)$ of the outgoing resolvent of the free Laplacian $-\Delta$, is defined as follows:
\begin{equation}\label{eq:prelim-K-volume-operator-def}
    \K_D^\omega : L^2(D) \rightarrow H^2(D),
    \qquad
    \K_D^\omega[\phi](x) := \int_D G^\omega(x-y) \phi(y) \dd y,
\end{equation}
which satisfies $-(\Delta +\omega^2)\K_D^\omega[\phi] = \chi_D \phi$ in $\R^3$ and admits the following expansion:
\begin{equation}\label{eq:prelim-K-volume-operator-expansion}
    \K_D^\omega = \K_D + \sum_{n=1}^{\infty}\omega^n \K_{D,n},
    \qquad
    \K_{D,n}[\phi]:= \int_D G_n(x-y)\phi(y)\dd y,
\end{equation}
where
\begin{equation}
    G_n(x):=\frac{\iu^n\abs{x}^{n-1}}{4\pi n!},
\end{equation}
$\K_D$ is the Newtonian potential, and the series converges in the operator norm. In particular, for each bounded set $B \Subset \C$, there exists a constant $C_B >0$ such that
\begin{equation}\label{eq:prelim-bounding-diff-newtonian-potential}
    \norm{\K_D^{\epsilon\hat\omega} - \K_D}_{L^2(D) \rightarrow H^2(D)} \leq C_B \epsilon, \quad \hat\omega \in B,
\end{equation}
for all sufficiently small $\epsilon>0$. 

As in \cite{ammari_dielectric_2025}, we reformulate \eqref{eq:nonlinear-helmholtz} as a Lippmann--Schwinger equation. More precisely, let $u^{\mathrm{inc}} \in H^2_{\mathrm{loc}}(\R^3)$ be a prescribed incident field at frequency $\omega \in \C$. We seek a (nontrivial) \emph{nonlinear scattering solution} $0 \not\equiv u = u^{\mathrm{inc}}\restrict{D} + u^{\mathrm{sc}} \in H^2(D)$ such that
    \begin{equation}\label{eq:nonlin-LS-incfield-def}
        u = u^{\mathrm{inc}}\restrict{D} + \tau\omega^2\K_D^\omega[q_{\eta}(u) u] \quad \mathrm{in} \ D,
    \end{equation}
    where $q_\eta(u) = 1 + \eta \abs{u}^2$. In the subwavelength regime, the frequency satisfies $\omega(\tau) \rightarrow 0$ as $\tau \rightarrow \infty$.

\begin{proposition}
    Let $u \in H^2(D)$ be a solution to the nonlinear Lippmann--Schwinger equation \eqref{eq:nonlin-LS-incfield-def}. Then it can be extended to a solution $u \in H^2_{\mathrm{loc}}(\R^3)$ of the nonlinear Helmholtz equation \eqref{eq:nonlinear-helmholtz}.
\end{proposition}
\begin{proof}
    Define the full-space extension of the Lippmann--Schwinger solution $u\in H^2(D)$ as
    \begin{equation}
        u^{\mathrm{tot}}(x) := u^{\mathrm{inc}}(x) + \tau\omega^2\K^\omega[\chi_D q_\eta(u)u], \quad x\in \R^3,
    \end{equation}
    where $\K^\omega$ is the outgoing resolvent of the free Laplacian:
    \begin{equation}
        \K^\omega : L^2_{\mathrm{comp}}(\R^3) \rightarrow H^2_{\mathrm{loc}}(\R^3),
        \qquad
        \K^\omega[\phi](x):= \int_{\R^3}G^\omega(x-y)\phi(y)\dd y.
    \end{equation}
    It follows that
    \begin{equation}
        (-\Delta - \omega^2)u^{\mathrm{tot}} = \tau \omega^2 \chi_D q_\eta(u)u \quad \mathrm{in} \ \R^3,
    \end{equation}
    and $u^{\mathrm{sc}} = u^{\mathrm{tot}} - u^{\mathrm{inc}} = \tau \omega^2 \K^\omega[\chi_D q_\eta(u)u]$ is outgoing by construction. Finally, the restriction of $u^{\mathrm{tot}}$ to $D$ coincides with the original Lippmann--Schwinger solution by \eqref{eq:nonlin-LS-incfield-def}.
\end{proof}

Throughout, we denote by $\hat\omega$ the scaled frequency $\hat\omega := \epsilon^{-1} \omega$, where $\epsilon:= \tau^{-1/2}$. Our interest lies in the scaled nonlinear Lippmann--Schwinger equation in the presence of an incident field $u^{\mathrm{inc}}$, namely
\begin{equation}\label{eq:prelim-nonlin-LS-incfield}
    u = u^{\mathrm{inc}} + \hat\omega^2 \K_D^{\epsilon\hat\omega}[u + \eta\abs{u}^2 u], \quad \mathrm{in} \ D.
\end{equation}
For conciseness, we define
\begin{equation}
    F(u,\hat\omega,\epsilon,\eta,u^{\mathrm{inc}})
    :=u-u^{\mathrm{inc}}
    -\hat\omega^2\K_D^{\epsilon\hat\omega}[u+\eta\abs{u}^2u],
\end{equation}
so that equation \eqref{eq:prelim-nonlin-LS-incfield} can be written compactly as
\begin{equation}\label{eq:prelim-nonlin-LS-compact-notation}
    F(u,\hat\omega,\epsilon,\eta,u^{\mathrm{inc}}) = 0.
\end{equation}
Here, $u$ denotes the total field inside $D$, and the corresponding exterior scattered field is
\begin{equation}\label{eq:scattered-field-equation}
    u^{\mathrm{sc}}(x) = \hat\omega^2\int_D G^{\epsilon\hat\omega}(x-y)
    \big(u(y)+\eta\abs{u(y)}^2u(y)\big)\dd y,
    \quad x \in \R^3 \setminus \overline{D}.
\end{equation}
When $\eta =1$, the nonlinearity coincides with the map $N(u): H^2(D) \rightarrow L^2(D)$ used in \cite{ammari_dielectric_2025}. Moreover, when $u^{\mathrm{inc}}=0$, we recover the nonlinear Lippmann--Schwinger equation from \cite[Eq. 3.3]{ammari_dielectric_2025}. Near a simple linear resonance, the cited work proves the existence and local uniqueness up to a phase of a small-amplitude nonlinear branch after restricting the equation to the level set
\begin{equation}\label{eq:prelim-level-set}
    \norm{u}^2_{L^2(D)} = \mathcal{N},
\end{equation}
where $\mathcal{N}>0$ is the normalization constant.

\subsection{Local Reduction}
From the spectral decomposition of the Newtonian potential, we know that
\begin{equation}
    \K_D[u] = \sum_{j=0}^{\infty}\lambda_j \innerproduct{\phi_j}{u}_D \phi_j,
\end{equation}
where $\lambda_0 > \lambda_1 \geq \lambda_2 \geq \cdots$, counting multiplicities, and $\norm{\phi_j}_{L^2(D)}=1$. Fix a simple eigenvalue $\lambda_j>0$ and let $\phi_j$ be its normalized eigenfunction:
\begin{equation}
    \K_D \phi_j = \lambda_j \phi_j, \quad \norm{\phi_j}_{L^2(D)}=1.
\end{equation}
Since the eigenvalue is nonzero, the definition of the Newtonian potential implies that $\phi_j \in H^2(D)$. Additionally, let $\hat\omega_j := \lambda_j^{-1/2}$.
Let us now define the projections 
\begin{equation}
\begin{aligned}
    &P_j u := \innerproduct{\phi_j}{u}_D \phi_j, \quad Q_j := I - P_j, \\
    & \ker P_j \cap H^2(D) = \{ h\in H^2(D) : \innerproduct{\phi_j}{h}_D = 0\},
\end{aligned}
\end{equation}
where $H^2(D) = \operatorname{span}\{\phi_j\} \oplus (\ker P_j \cap H^2(D))$. Thus, every $u\in H^2(D)$ can be written uniquely as
\begin{equation}\label{eq:decomposition-u}
    u = a\phi_j + h, \quad a = \innerproduct{\phi_j}{u}_D \in \C, \quad h = Q_j u \in \ker P_j \cap H^2(D).
\end{equation}
Before applying the reduction to our nonlinear problem with an incident field, we state and prove the following key lemmas. For conciseness, we denote $X_p := \ker P_j \cap H^2(D)$.
\begin{lemma}\label{lem:complement-newtonian-potential}
    Let $\lambda_j$ be simple. Then there exists a small neighbourhood $U \subset \C$ of $\hat\omega_j$ and a constant $C_0 >0$ such that, for every $\hat\omega \in U$,
    \begin{equation}
        (I-\hat\omega^2 \K_D) : X_p \rightarrow X_p
    \end{equation}
    is invertible and 
    \begin{equation}
        \norm{(I-\hat\omega^2 \K_D)^{-1}}_{X_p \rightarrow X_p} \leq C_0.
    \end{equation}
\end{lemma}
\begin{proof}
    First, the map is well defined by the self-adjointness of $\K_D$ on $L^2(D)$. Indeed, if $h \in X_p$, then $\innerproduct{\phi_j}{\K_D h}_D = \lambda_j \innerproduct{\phi_j}{h}_D = 0$.
    Since $\lambda_j$ is an isolated simple eigenvalue, we can choose a neighbourhood $U \subset \C$ of $\hat\omega_j := \lambda_j^{-1/2}$ such that $\overline{U}\cap \{\pm \lambda_k^{-1/2} : k \neq j\} = \varnothing$. Since $\lambda_k \rightarrow0$ as $k\rightarrow \infty$ and $\overline{U}$ is bounded, it follows that
    \begin{equation}
         \inf_{\hat\omega \in \overline{U}} \inf_{k\neq j} \abs{1 - \hat\omega^2 \lambda_k} > 0, 
    \end{equation}
	   where $\{ \lambda_k \}$ denotes the nonzero eigenvalues of $\K_D$, repeated according to their multiplicity. Let $g\in X_p$ and consider solving $(I-\hat\omega^2\K_D)h = g$ in $L^2(D)$. By the spectral theorem,
    \begin{equation}
        L^2(D) = \ker\K_D \oplus \overline{\operatorname{span}\{\phi_k : \lambda_k \neq 0\}},
    \end{equation}
    and, since $\ker\K_D \perp \phi_j$, every $g\in X_p$ can be written as
    \begin{equation}
        g = g_0 + \sum_{k\neq j}{g_k \phi_k}, \quad g_0 \in \ker\K_D.
    \end{equation}
    Hence, we define
    \begin{equation}
        h := g_0 + \sum_{k \neq j}{\frac{g_k}{1-\hat\omega^2 \lambda_k}\phi_k}, 
    \end{equation}
    where the series converges in $L^2(D)$ and satisfies
    \begin{equation}
        \norm{h}_{L^2(D)} \leq \dfrac{\norm{g}_{L^2(D)}}{\inf_{\hat\omega \in U} \inf_{k\neq j} \abs{1 - \hat\omega^2 \lambda_k}}.
    \end{equation}
    Finally, the identity $h=g+\hat\omega^2\K_Dh$, together with $g\in H^2(D)$ and $\K_Dh\in H^2(D)$ for $h\in L^2(D)$, shows that $h\in H^2(D)$ and hence $h\in X_p$. Since $U$ is bounded, we also obtain
    \begin{equation}
        \norm{h}_{H^2(D)} \leq \norm{g}_{H^2(D)} + C_1 \norm{h}_{L^2(D)} \leq C \norm{g}_{H^2(D)},
    \end{equation}
    for some positive constants $C_1, C >0$, with $C$ independent of $\hat\omega \in U$. Hence, $(I- \hat\omega^2\K_D)^{-1}:X_p \rightarrow X_p$ exists and is uniformly bounded for $\hat\omega \in U$.
    
\end{proof}

Next, we establish the analogous statement for the volume operator $\K_D^\omega$.
\begin{lemma}\label{lem:complement-K-d-omega}
    Suppose that $\lambda_j$ is a simple eigenvalue. Then there exist a small neighbourhood $U \subset \C$ of $\hat\omega_j$, a number $\epsilon_0 > 0$, and a constant $C>0$ such that, for every $0 < \epsilon \leq \epsilon_0$ and $\hat\omega \in U$,
    \begin{equation}
        Q_j(I - \hat\omega^2 \K_D^{\epsilon\hat\omega}) : X_p \rightarrow X_p
    \end{equation}
    is invertible and 
    \begin{equation}
        \norm{\big( Q_j(I-\hat\omega^2\K_D^{\epsilon\hat\omega}) \big)^{-1}}_{X_p \rightarrow X_p} \leq C.
    \end{equation}
\end{lemma}
\begin{proof}
    Using the expansion of $\K_D^\omega$, we write
    \begin{equation}
        Q_j(I-\hat\omega^2\K_D^{\epsilon\hat\omega}) = Q_j(I-\hat\omega^2\K_D)  - \hat\omega^2Q_j(\K_D^{\epsilon\hat\omega} - \K_D)\restrict{X_p}.
    \end{equation}
    Since $U$ is bounded and $Q_j$ is a bounded projection on $H^2(D)$, \eqref{eq:prelim-bounding-diff-newtonian-potential} implies that there exists a constant $C_B >0$ such that
    \begin{equation}
        \norm{- \hat\omega^2Q_j(\K_D^{\epsilon\hat\omega} - \K_D)}_{X_p \rightarrow X_p} \leq C_B \epsilon, \quad \hat\omega \in U.
    \end{equation}
    Moreover, $Q_j(I-\hat\omega^2\K_D) \equiv (I-\hat\omega^2\K_D)\restrict{X_p}$, and Lemma \ref{lem:complement-newtonian-potential} yields
    \begin{equation}\label{eq:complicated-term}
        \norm{(Q_j(I-\hat\omega^2\K_D))^{-1} (- \hat\omega^2Q_j(\K_D^{\epsilon\hat\omega} - \K_D))}_{X_p \rightarrow X_p} \leq C_0 C_B \epsilon.
    \end{equation}
    Thus, we can choose $\epsilon_0 > 0$ such that $C_0C_B\epsilon_0 \leq \frac{1}{2}$ and denote the operator inside the norm in \eqref{eq:complicated-term} by $A$. It follows that $I + A$ is invertible by a Neumann series and
    \begin{equation}
        (Q_j(I-\hat\omega^2\K_D^{\epsilon\hat\omega}))^{-1} = (I + A)^{-1}(Q_j(I-\hat\omega^2\K_D))^{-1},
    \end{equation}
   and therefore $\norm{(Q_j(I-\hat\omega^2\K_D^{\epsilon\hat\omega}))^{-1}}_{X_p \rightarrow X_p} \leq 2 C_0$, uniformly in $\hat\omega \in U$ and $0 < \epsilon \leq \epsilon_0$.
\end{proof}
We now use the decomposition $u = a\phi_j + h$ and apply the projection $Q_j$ to equation \eqref{eq:prelim-nonlin-LS-incfield} to obtain
\begin{equation}\label{eq:Qprojected-incident-field-equation}
\begin{aligned}
    &Q_j(u - u^{\mathrm{inc}} - \hat\omega^2\K_D^{\epsilon\hat\omega}[u + \eta\abs{u}^2u]) = 0 \\
    \iff &Q_j(I-\hat\omega^2\K_D^{\epsilon\hat\omega})[h] = Q_j u^{\mathrm{inc}} + \hat\omega^2 Q_j\K_D^{\epsilon\hat\omega}[a \phi_j] +\eta\hat\omega^2Q_j \K_D^{\epsilon\hat\omega}[ \abs{a\phi_j + h}^2 (a\phi_j +h)].
\end{aligned}
\end{equation}
\begin{remark}\label{rmk:bound-easy-term}
    The term involving $a\phi_j$ is small because $Q_j\K_D[a\phi_j] = a Q_j(\lambda_j \phi_j) = 0$. Consequently,
    \begin{equation}
        \norm{Q_j\K_D^{\epsilon\hat\omega}[a\phi_j]}_{H^2(D)}
        = \norm{aQ_j(\K_D^{\epsilon\hat\omega}-\K_D)\phi_j}_{H^2(D)}
        \leq C \epsilon \abs{a}
    \end{equation}
    for some constant $C>0$.
\end{remark}
We are now ready to formulate a local decomposition result using Lyapunov--Schmidt reduction.

\begin{proposition}\label{prop:local-reduction}
    Consider, as before, a small neighbourhood $U \subset \C$ of $\hat\omega_j$, and fix $\eta_{\mathrm{max}}>0$. There exist $r,R>0$ and $\epsilon_0>0$ such that the following holds. For every $0<\eta<\eta_{\mathrm{max}}$, $0<\epsilon<\epsilon_0$, $\hat\omega \in U$, and $\abs{a} +\norm{u^{\mathrm{inc}}}_{H^2(D)} \leq r$, equation \eqref{eq:Qprojected-incident-field-equation} has a unique solution $h = W_\eta(a, \hat\omega, \epsilon, u^{\mathrm{inc}}) \in X_p$ in the ball $\norm{h}_{H^2(D)} \leq R$. This solution map is bounded on the parameter neighbourhood and jointly real analytic in $(a,\hat\omega,\epsilon,\eta,u^{\mathrm{inc}})$. Moreover, define
    \begin{equation}\label{eq:def-reduced-eq}
    G_j(a,\hat\omega,\epsilon,\eta,u^{\mathrm{inc}})
    :=\innerproduct{\phi_j}
    {F(a\phi_j+W_\eta(a,\hat\omega,\epsilon,u^{\mathrm{inc}}),
    \hat\omega,\epsilon,\eta,u^{\mathrm{inc}})}_D.
    \end{equation}
    Then every solution of equation \eqref{eq:prelim-nonlin-LS-incfield} satisfying the above parameter constraints and $\norm{Q_ju}_{H^2(D)}\leq R$ is uniquely of the form
    \begin{equation}\label{eq:decomp-u-into-reduced-W}
        u = a\phi_j + W_\eta(a,\hat\omega,\epsilon, u^{\mathrm{inc}}),
    \end{equation}
    and satisfies
    \begin{equation}\label{eq:G-reduced-functional}
        G_j(a,\hat\omega,\epsilon,\eta,u^{\mathrm{inc}}) = 0.
    \end{equation}
    For brevity, write $W_\eta:=W_\eta(a,\hat\omega,\epsilon,u^{\mathrm{inc}})$. Expanding this definition gives the following explicit nonlinear scalar equation for the amplitude coefficient:
    \begin{equation}\label{eq:amplitude-eq}
        \begin{aligned}
        &\left(1 - \hat\omega^2
        \innerproduct{\phi_j}{\K_D^{\epsilon\hat\omega}[\phi_j]}_D\right)a\\
        &\quad= \innerproduct{\phi_j}{u^{\mathrm{inc}}}_D
        + \hat\omega^2\innerproduct{\phi_j}{\K_D^{\epsilon\hat\omega}[W_\eta]}_D\\
        &\qquad
        + \hat\omega^2\eta\innerproduct{\phi_j}{\K_D^{\epsilon\hat\omega}
        [\abs{a\phi_j + W_\eta}^2 (a\phi_j + W_\eta)]}_D.
        \end{aligned}
    \end{equation}
    Conversely, every solution $a$ of \eqref{eq:G-reduced-functional} satisfying the stated constraints yields a solution of \eqref{eq:prelim-nonlin-LS-incfield} through \eqref{eq:decomp-u-into-reduced-W}.
\end{proposition}
\begin{proof}
    Consider the neighbourhood $U_{\hat\omega_j} := \{ \hat\omega \in \C : \abs{\hat\omega - \hat\omega_j} < \delta \}$ for a sufficiently small $\delta >0$. Denote the left-hand side of \eqref{eq:Qprojected-incident-field-equation} by
    \begin{equation}
        \mathcal{A}(\hat\omega, a, h, \epsilon,\eta,u^{\mathrm{inc}})
        := Q_j(a\phi_j + h) - Q_ju^{\mathrm{inc}}
        - \hat\omega^2 Q_j\K_D^{\epsilon\hat\omega}
        [a\phi_j + h + \eta \abs{a\phi_j+h}^2(a\phi_j+h)],
    \end{equation}
    where $\mathcal{A}$ is real analytic as a map
    \begin{equation}
        \mathcal{A}:U_{\hat\omega_j}\times\C\times X_p\times\R\times\R
        \times H^2(D)\longrightarrow X_p.
    \end{equation}
    For every $\eta\in[0,\eta_{\mathrm{max}}]$, we have
    \begin{equation}
    \begin{aligned}
        \mathcal{A}(\hat\omega_j,0,0,0,\eta,0) &= 0,\\
        \partial_h\mathcal{A}(\hat\omega_j,0,0,0,\eta,0)[v]
        &= Q_j(I-\hat\omega_j^2\K_D)[v]
        = (I-\hat\omega_j^2\K_D)[v], \qquad v \in X_p.
        \end{aligned}
    \end{equation}
    Lemma \ref{lem:complement-newtonian-potential} shows that this derivative is invertible and independent of $\eta$. The parameter-dependent real-analytic implicit function theorem gives a local solution for every $\eta$. By the joint analyticity of $\mathcal{A}$, the compactness of $[0,\eta_{\mathrm{max}}]$, and the fact that the base derivative is independent of $\eta$, the parameter neighbourhood can be chosen uniformly in $\eta$. We therefore obtain a common neighbourhood $U_{(0,\hat\omega_j,0,0)}$ of $(0,\hat\omega_j,0,0)$ and a family
    \begin{equation}
        W:U_{(0,\hat\omega_j,0,0)}\times[0,\eta_{\mathrm{max}}]
        \longrightarrow X_p,
        \qquad
        W_\eta(a,\hat\omega,\epsilon,u^{\mathrm{inc}})
        :=W(a,\hat\omega,\epsilon,u^{\mathrm{inc}},\eta).
    \end{equation}
    For the remainder of the proof, write $W_\eta:=W_\eta(a,\hat\omega,\epsilon,u^{\mathrm{inc}})$.
    The map extends real-analytically to an open neighbourhood of its domain and satisfies
    \begin{equation}
        \mathcal{A}(\hat\omega,a,W_\eta,\epsilon,\eta,u^{\mathrm{inc}})=0.
    \end{equation}
    Thus, \eqref{eq:Qprojected-incident-field-equation} admits a unique local solution $h=W_\eta(a,\hat\omega,\epsilon,u^{\mathrm{inc}})$ for parameters close to $(0,\hat\omega_j,0,0)$, uniformly for $\eta\in[0,\eta_{\mathrm{max}}]$. Finally, \eqref{eq:prelim-nonlin-LS-incfield} is equivalent to the projected system $Q_jF=0$ and $P_jF=0$. By local uniqueness, the first equation is equivalent to $h=W_\eta(a,\hat\omega,\epsilon,u^{\mathrm{inc}})$, whereas the second is equivalent to
    \begin{equation}\label{eq:Pj-equation}
        \innerproduct{\phi_j}{u - u^\mathrm{inc} - \hat\omega^2\K_D^{\epsilon\hat\omega}[u + \eta \abs{u}^2 u]}_D = 0.
    \end{equation}
    By substituting the decomposition $u= a\phi_j + W_\eta$ into \eqref{eq:Pj-equation} and using that $\innerproduct{\phi_j}{W_\eta}_D = 0$, we obtain the nonlinear equation \eqref{eq:amplitude-eq} for the coefficient $a$.
\end{proof}

\begin{remark}
    Proposition \ref{prop:local-reduction} applies when both the incident wave $u^\mathrm{inc}$ and the resonant amplitude $a$ are sufficiently small, while the nonlinearity strength remains bounded, $\eta=O(1)$. Here $\eta>0$ is fixed independently of $\epsilon$, $a$, and $\hat\omega$, and the construction is uniform for $0<\eta<\eta_{\mathrm{max}}$. It is also useful to recast the proposition in a second regime, in which the incident wave is of order one and $\eta$ is arbitrarily small.
\end{remark}

\begin{proposition}\label{prop:alternative-regime-local-reduction}
    Let $U\subset\C$ be a sufficiently small neighbourhood of $\hat\omega_j$. For every fixed $r, M >0$, there exist $\eta_0 >0$, $\epsilon_0 >0$, a constant $C>0$, and a radius $R>0$ such that the following holds. For every $\abs{a}<r$, $\norm{u^{\mathrm{inc}}}_{H^2(D)}<M$, $\hat\omega\in U$, $0<\epsilon\leq \epsilon_0$, and $\abs{\eta} \leq \eta_0$, equation \eqref{eq:Qprojected-incident-field-equation} has a unique solution $h=W_{\eta}(a,\hat\omega,\epsilon,u^{\mathrm{inc}}) \in X_p$ in the ball $\norm{h}_{H^2(D)} \leq R$. The corresponding linear solution is
    \begin{equation}
        h_{\mathrm{lin}} = \big( Q_j(I-\hat\omega^2 \K_D^{\epsilon\hat\omega})\big)^{-1} \big( Q_j u^{\mathrm{inc}} + \hat\omega^2 Q_j\K_D^{\epsilon\hat\omega}[a\phi_j] \big).
    \end{equation}
    Moreover, the map $W_\eta$ satisfies
    \begin{equation}
        \norm{W_{\eta} - h_{\mathrm{lin}}}_{H^2(D)} \leq C \abs{\eta}(\norm{Q_j u^{\mathrm{inc}}}_{H^2(D)} + \abs{a})^3.
    \end{equation}
    Define
    \begin{equation}
        G_j(a,\hat\omega,\epsilon,\eta,u^{\mathrm{inc}})
        :=\innerproduct{\phi_j}
        {F(a\phi_j+W_\eta(a,\hat\omega,\epsilon,u^{\mathrm{inc}}),
        \hat\omega,\epsilon,\eta,u^{\mathrm{inc}})}_{D}.
    \end{equation}
    Then every solution $u$ of equation \eqref{eq:prelim-nonlin-LS-incfield} satisfying $\abs{\innerproduct{\phi_j}{u}_D} < r$ and $\norm{Q_j u }_{H^2(D)} \leq M$ is uniquely of the form
    \begin{equation}
        u = a\phi_j + W_\eta(a,\hat\omega, \epsilon,u^{\mathrm{inc}}),
    \end{equation}
    where $a$ satisfies $G_j(a,\hat\omega,\epsilon,\eta,u^{\mathrm{inc}})=0$.
\end{proposition}
\begin{proof}
    Fix $(a,\hat\omega, \epsilon, u^{\mathrm{inc}}) \in \C \times U \times \R \times H^2(D)$, decompose $u=a\phi_j+h$ as in \eqref{eq:decomposition-u}, and consider the linear case of \eqref{eq:Qprojected-incident-field-equation}, namely $\eta=0$:
    \begin{equation}\label{eq:linear-Q-proj-equation}
        Q_j(I-\hat\omega^2\K_D^{\epsilon\hat\omega})[h] = Q_j u^{\mathrm{inc}} + \hat\omega^2 Q_j\K_D^{\epsilon\hat\omega}[a\phi_j].
    \end{equation}
    Lemma \ref{lem:complement-K-d-omega} shows that the left-hand side is uniformly invertible for $\hat\omega \in U$ and sufficiently small $0<\epsilon \leq \epsilon_0$. Thus, the solution of \eqref{eq:linear-Q-proj-equation} is
    \begin{equation}
        h_{\mathrm{lin}} = \big(Q_j(I-\hat\omega^2\K_D^{\epsilon\hat\omega})\big)^{-1}
        \big(Q_ju^{\mathrm{inc}} + \hat\omega^2 Q_j\K_D^{\epsilon\hat\omega}[a\phi_j]\big).
    \end{equation}
    By the same lemma and Remark \ref{rmk:bound-easy-term}, there is a constant $C_0>0$ such that
    \begin{equation}\label{eq:bound-hlin}
        \norm{h_{\mathrm{lin}}}_{H^2(D)}
        \leq C_0\big(\norm{Q_ju^{\mathrm{inc}}}_{H^2(D)}+\epsilon\abs{a}\big)
        \leq C_0(M+r).
    \end{equation}
    Define the map $\mathcal{A}:X_p\rightarrow X_p$ by
    \begin{equation}
        \mathcal{A}(h) := h_{\mathrm{lin}}
        + \eta \big(Q_j(I-\hat\omega^2\K_D^{\epsilon\hat\omega})\big)^{-1}
        \big(\hat\omega^2 Q_j \K_D^{\epsilon\hat\omega}[\abs{a\phi_j + h}^2 (a\phi_j +h)]\big).
    \end{equation}
    Then $h$ solves equation \eqref{eq:Qprojected-incident-field-equation} if and only if $h = \mathcal{A}(h)$. In the following, consider the closed ball $B(0,R) := \{h\in X_p: \norm{h}_{H^2(D)} \leq R\}$ and the decomposition $u=a\phi_j +h$, for which the following estimate in dimension three holds due to the Sobolev embedding:
    \begin{equation}
        \norm{\abs{u}^2 u}_{L^2(D)} \leq \norm{u}^2_{L^\infty(D)}\norm{u}_{L^2(D)} \leq C \norm{u}^3_{H^2(D)}.
    \end{equation}
    Hence, for some constant $C>0$, if $\abs{a}<r$ and $h\in B(0,R)$, then
    \begin{equation}
        \norm{\abs{a\phi_j +h}^2 (a\phi_j+h)}_{L^2(D)} \leq C(r+R)^3.
    \end{equation}
    Moreover, for $h_1,h_2\in B(0,R)$, set $u_1=a\phi_j+h_1$ and $u_2=a\phi_j+h_2$. Then
    \begin{equation}
        \abs{\abs{u_1}^2u_1-\abs{u_2}^2u_2}
        \leq C(\abs{u_1}^2+\abs{u_2}^2)\abs{u_1-u_2},
    \end{equation}
    and consequently
    \begin{equation}
        \norm{\abs{a\phi_j + h_1}^2 (a\phi_j + h_1) - \abs{a\phi_j+h_2}^2(a\phi_j+h_2)}_{L^2(D)} \leq C (r+R)^2 \norm{h_1 - h_2}_{H^2(D)}.
    \end{equation}
    We can thus choose $R:=\max\{M,C_0(M+r)\}+1$, where $C_0$ is the constant in \eqref{eq:bound-hlin}, so that $\norm{h_{\mathrm{lin}}}_{H^2(D)}\leq R-1$. Then, for $h\in B(0,R)$,
    \begin{equation}
        \norm{\mathcal{A}(h)}_{H^2(D)} \leq \norm{h_{\mathrm{lin}}}_{H^2(D)} + C \abs{\eta} (r+R)^3 \leq R - 1 + C \abs{\eta}(r+R)^3.
    \end{equation}
    If $\eta_0 >0$ is chosen small enough that $C \eta_0 (r+R)^3 \leq 1$, then the map $\mathcal{A}$ maps the closed ball $B(0,R)$ into itself. Finally, given $h_1,h_2 \in B(0,R)$, we have
    \begin{equation}\label{eq:contraction-estimate}
        \norm{\mathcal{A}(h_1) - \mathcal{A}(h_2)}_{H^2(D)} \leq C \abs{\eta}(r+R)^2 \norm{h_1-h_2}_{H^2(D)},
    \end{equation}
    Decreasing $\eta_0$ further, if necessary, so that $C \eta_0 (r+R)^2 \leq \frac{1}{2}$, we find that $\mathcal{A}$ is a contraction on the closed ball. Applying the Banach fixed point theorem gives the unique fixed point $h = W_\eta(a,\hat\omega,\epsilon,u^{\mathrm{inc}}) \in X_p$ with $\norm{h}_{H^2(D)} \leq R$.
    The fixed-point equation and \eqref{eq:bound-hlin} first give
    \begin{equation}
        \norm{W_\eta}_{H^2(D)}
        \leq C\big(\norm{Q_ju^{\mathrm{inc}}}_{H^2(D)}+\abs{a}\big)
        + C\abs{\eta}(r+R)^2\norm{W_\eta}_{H^2(D)}.
    \end{equation}
    Decreasing $\eta_0$ if necessary, we can absorb the last term and obtain
    $\norm{W_\eta}_{H^2(D)}\leq C(\norm{Q_ju^{\mathrm{inc}}}_{H^2(D)}+\abs{a})$. Therefore,
    \begin{equation}
        \norm{W_\eta - h_{\mathrm{lin}}}_{H^2(D)}
        \leq C \abs{\eta}(\abs{a} + \norm{W_\eta}_{H^2(D)})^3
        \leq C \abs{\eta}\big(\norm{Q_ju^{\mathrm{inc}}}_{H^2(D)}+\abs{a}\big)^3.
    \end{equation}
    Define the real-analytic map
    \begin{equation}
        \mathcal{F}(h,a,\hat\omega,\epsilon,\eta,u^{\mathrm{inc}})
        :=h-\mathcal{A}(h),
    \end{equation}
    from $X_p \times \C \times U \times \R \times \R \times H^2(D)$ to $X_p$. The contraction estimate \eqref{eq:contraction-estimate} gives
    $\norm{D_h\mathcal{A}}_{X_p\to X_p}\leq\frac12$, so
    $D_h\mathcal{F}=I-D_h\mathcal{A}$ is invertible by a Neumann-series argument. The implicit function theorem then shows that $W_\eta$ is real analytic in $(a,\hat\omega,\epsilon,\eta,u^{\mathrm{inc}})$.
\end{proof}

\subsection{Amplitude Equation}\label{sec:amplitude-eq}

Using the reduced equation \eqref{eq:G-reduced-functional} from Proposition \ref{prop:local-reduction}, we obtain a scalar nonlinear equation for the resonant coefficient. Consider a simple eigenpair $(\lambda_j,\phi_j)$ of $\K_D$ and, for brevity, write $W:=W_\eta(a,\hat\omega,\epsilon,u^{\mathrm{inc}})$. We first establish two bounds on this remainder.

\begin{lemma}\label{lem:helper-W-bound}
    Under the assumptions of Proposition \ref{prop:local-reduction}, consider the compact neighbourhood $V := \overline{U_{\hat\omega_j}} \times [0,\epsilon_0] \times [0,\eta_{\mathrm{max}}]$, contained in the parameter neighbourhood on which the real-analytic extension of $W_\eta$ is defined. Then
    \begin{equation}
        \lim_{s\rightarrow0^+} \sup_{\substack{(\hat\omega,\epsilon,\eta) \in V \\ \abs{a} + \norm{u^{\mathrm{inc}}}_{H^2(D)} < s}} \norm{W_\eta (a,\hat\omega, \epsilon, u^{\mathrm{inc}})}_{H^2(D)} = 0.
    \end{equation}
\end{lemma}
\begin{proof}
    The construction in the proof of Proposition \ref{prop:local-reduction} shows that $W_\eta(a,\hat\omega,\epsilon,u^{\mathrm{inc}})$ is jointly continuous on a neighbourhood of each point $(0,\hat\omega,\epsilon,0,\eta)$ with $(\hat\omega,\epsilon,\eta)\in V$. Moreover,
    \begin{equation}\label{eq:trivial-h}
        W_\eta(0,\hat\omega,\epsilon,0) = 0,
        \qquad \text{for every }(\hat\omega,\epsilon,\eta)\in V,
    \end{equation}
    because $h=0$ solves \eqref{eq:Qprojected-incident-field-equation} when $a=0$ and $u^{\mathrm{inc}}=0$; local uniqueness then gives \eqref{eq:trivial-h}.
    Fix $\rho>0$. For every $y\in V$, joint continuity and \eqref{eq:trivial-h} provide a constant $C_y>0$ and a neighbourhood $V_y$ of $y$ in $V$ such that
    \begin{equation}
        \norm{W_{\eta'}(a,\hat\omega',\epsilon',u^{\mathrm{inc}})}_{H^2(D)}<\rho
    \end{equation}
    whenever $(\hat\omega',\epsilon',\eta')\in V_y$ and
    $\abs{a}+\norm{u^{\mathrm{inc}}}_{H^2(D)}<C_y$.
    The sets $V_y$ cover the compact set $V$. Choose a finite subcover
    $V_{y_1},\ldots,V_{y_N}$ and set
    $C_\rho:=\min_{1\leq i\leq N}C_{y_i}>0$.
    The desired bound then holds uniformly over $V$.
\end{proof}

\begin{lemma}\label{lem:uniform-W-bound}
    Under the assumptions of Proposition \ref{prop:local-reduction}, there exists a constant $C_W > 0$ such that
    \begin{equation}
        \norm{W}_{H^2(D)} \leq C_W \big( \norm{Q_j u^{\mathrm{inc}}}_{H^2(D)} + \epsilon \abs{a} + \eta \abs{a}^3 \big),
    \end{equation}
    for every tuple $(a,\hat\omega, \epsilon,u^{\mathrm{inc}},\eta)$ satisfying the parameter constraints from Proposition \ref{prop:local-reduction}. In particular, the constant is uniform for $\hat\omega \in U_{\hat\omega_j}$, $0<\epsilon< \epsilon_0$, and $0 < \eta < \eta_{\mathrm{max}}$.
\end{lemma}
\begin{proof}
    Consider the neighbourhood $U_{\hat\omega_j} := \{\hat\omega \in \C : \abs{\hat\omega - \hat\omega_j} < \delta \}$ for sufficiently small $\delta>0$. By Lemma \ref{lem:complement-K-d-omega}, there exists a constant $C_1>0$ such that
    \begin{equation}\label{eq:bound1-proofW}
        \sup_{\substack{\hat\omega \in U_{\hat\omega_j} \\ 0<\epsilon\leq \epsilon_0}}\norm{(Q_j (I - \hat\omega^2 \K_D^{\epsilon\hat\omega}))^{-1}}_{X_p \rightarrow X_p} \leq C_1.
    \end{equation}
    Since $U_{\hat\omega_j}$ is bounded, the operator expansion from \eqref{eq:prelim-K-volume-operator-expansion} gives
    \begin{equation}\label{eq:bound2-proofW}
        \sup_{\substack{\hat\omega \in U_{\hat\omega_j} \\ 0<\epsilon\leq \epsilon_0}}\norm{\hat\omega^2 Q_j \K_D^{\epsilon\hat\omega}}_{L^2(D) \rightarrow H^2(D)} \leq C_2,
    \end{equation}
    for some constant $C_2>0$. Finally, using \eqref{eq:prelim-bounding-diff-newtonian-potential} and Remark \ref{rmk:bound-easy-term}, we obtain
    \begin{equation}\label{eq:bound3-proofW}
        \norm{\hat\omega^2 Q_j \K_D^{\epsilon\hat\omega}[a\phi_j]}_{H^2(D)} = \abs{a} \norm{\hat\omega^2 Q_j(\K_D^{\epsilon\hat\omega} - \K_D)[\phi_j]}_{H^2(D)} \leq C_3 \epsilon \abs{a},
    \end{equation}
    for some positive constant $C_3 >0$.
    Now consider the projected equation \eqref{eq:Qprojected-incident-field-equation}, written as
    \begin{equation}\label{eq:Qj-proj-rewritten}
        Q_j(I-\hat\omega^2\K_D^{\epsilon\hat\omega})[W] = Q_j u^{\mathrm{inc}} + \hat\omega^2 Q_j\K_D^{\epsilon\hat\omega}[a \phi_j] +\eta\hat\omega^2Q_j \K_D^{\epsilon\hat\omega}[ \abs{a\phi_j + W}^2 (a\phi_j +W)].
    \end{equation}
    We apply the inverse $(Q_j(I-\hat\omega^2\K_D^{\epsilon\hat\omega}))^{-1}$ to \eqref{eq:Qj-proj-rewritten}, and using the bounds from \eqref{eq:bound1-proofW}, \eqref{eq:bound2-proofW}, and \eqref{eq:bound3-proofW} we obtain
    \begin{equation}\label{bound-w}
        \norm{W}_{H^2(D)} \leq C ( \norm{Q_j u^{\mathrm{inc}}}_{H^2(D)} + \epsilon \abs{a} + \eta (\abs{a}+\norm{W}_{H^2(D)})^3 ).
    \end{equation}
    We use the elementary inequality $(x+y)^3 \leq 4x^3 + 4y^3$ for nonnegative $x,y \in \R$ to rewrite \eqref{bound-w} as
    \begin{equation}\label{eq:rewritten-w-bound}
        \norm{W}_{H^2(D)} \leq C ( \norm{Q_j u^{\mathrm{inc}}}_{H^2(D)} + \epsilon \abs{a}) + 4C\eta \abs{a}^3 + 4C\eta \norm{W}_{H^2(D)}^3.
    \end{equation}
    Choose $\rho>0$ sufficiently small that $4C\eta_{\mathrm{max}}\rho^2\leq \frac{1}{2}$. By Lemma \ref{lem:helper-W-bound}, we may then decrease the radius $r$ in Proposition \ref{prop:local-reduction} so that $\norm{W}_{H^2(D)}<\rho$ for all admissible parameter tuples. We can therefore absorb the final term in \eqref{eq:rewritten-w-bound} to obtain
    \begin{equation}
        \norm{W}_{H^2(D)} \leq 2C ( \norm{Q_j u^{\mathrm{inc}}}_{H^2(D)} + \epsilon \abs{a}) + 8C\eta \abs{a}^3,
    \end{equation}
    which concludes the proof.
\end{proof}

We now consider the regime of Proposition \ref{prop:local-reduction}: the nonlinearity is bounded by $\eta_{\mathrm{max}}$, while the amplitude and incident field satisfy $\abs{a}+\norm{u^\mathrm{inc}}_{H^2(D)}\leq r$. We use the representation
\begin{equation}
    u = a\phi_j + W,
\end{equation}
in the scalar equation $G_j(a,\hat\omega,\epsilon,\eta,u^{\mathrm{inc}})=0$ to obtain
\begin{equation}\label{eq:leading-order-reduced-equation}
    \begin{aligned}
        0={}&a
        -\innerproduct{\phi_j}{u^{\mathrm{inc}}}_D
        -\hat\omega^2\innerproduct{\phi_j}{\K_D^{\epsilon\hat\omega}[a\phi_j + W]}_D \\
        &-\eta \hat\omega^2
        \innerproduct{\phi_j}{\K_D^{\epsilon\hat\omega}[\abs{a\phi_j+W}^2(a\phi_j+W)]}_D.
    \end{aligned}
\end{equation}
Using the linearity of $\K_D^{\epsilon\hat\omega}$, we separate the terms that do not contain $W$ as follows:
\begin{equation}\label{eq:unfolded-Pj-equation}
\begin{aligned}
    &a(1 - \hat\omega^2\innerproduct{\phi_j}{\K_D^{\epsilon\hat\omega}[\phi_j]}_D) - \innerproduct{\phi_j}{u^{\mathrm{inc}}}_D - \eta \abs{a}^2 a \hat\omega^2\innerproduct{\phi_j}{\K_D^{\epsilon\hat\omega}[\abs{\phi_j}^2 \phi_j]}_D \\
    &= \hat\omega^2\innerproduct{\phi_j}{\K_D^{\epsilon\hat\omega}[W]}_D + \eta\hat\omega^2\innerproduct{\phi_j}{\K_D^{\epsilon\hat\omega}[\abs{a\phi_j+W}^2(a\phi_j+W) - \abs{a}^2 a \abs{\phi_j}^2 \phi_j]}_D.
    \end{aligned}
\end{equation}
Let us consider the right-hand side of equation \eqref{eq:unfolded-Pj-equation}, which we can bound in this regime by the following lemma.
\begin{lemma}
    In the local regime of Proposition \ref{prop:local-reduction}, define
    \begin{equation}\label{eq:def-rest-term-W}
        \mathcal{R}_W := \hat\omega^2\innerproduct{\phi_j}{\K_D^{\epsilon\hat\omega}[W]}_D + \eta\hat\omega^2\innerproduct{\phi_j}{\K_D^{\epsilon\hat\omega}[\abs{a\phi_j+W}^2(a\phi_j+W) - \abs{a}^2 a \abs{\phi_j}^2 \phi_j]}_D.
    \end{equation}
    Then
\begin{equation}\label{eq:bound-rest-w}
    \abs{\mathcal{R}_W} \leq C \Big( \epsilon\norm{W}_{H^2(D)}
    + \abs{\eta}\big( \abs{a}^2\norm{W}_{H^2(D)}
    + \abs{a}\norm{W}_{H^2(D)}^2 + \norm{W}_{H^2(D)}^3\big)\Big),
\end{equation}
for some constant $C>0$.
\end{lemma}
\begin{proof}
    For the first term, we use self-adjointness of $\K_D$ together with the fact that $W \in X_p$ to observe that
    \begin{equation}\label{eq:vanishing-W-term}
        \innerproduct{\phi_j}{\K_D[W]}_D = \innerproduct{\K_D[\phi_j]}{W}_D = \lambda_j \innerproduct{\phi_j}{W}_D = 0.
    \end{equation}
    Combining \eqref{eq:vanishing-W-term} with the bound from \eqref{eq:prelim-bounding-diff-newtonian-potential}, we obtain
    \begin{equation}
    \begin{aligned}
        \abs{\hat\omega^2\innerproduct{\phi_j}{\K_D^{\epsilon\hat\omega}[W]}_D}
        &= \abs{\hat\omega^2\innerproduct{\phi_j}{(\K_D^{\epsilon\hat\omega} - \K_D)[W]}_D} \\
        &\leq \abs{\hat\omega}^2 \norm{\phi_j}_{L^2(D)}
        \norm{(\K_D^{\epsilon\hat\omega} - \K_D)[W]}_{L^2(D)} \\
        &\leq C \epsilon \norm{W}_{H^2(D)},
    \end{aligned}
    \end{equation}
    where the first inequality is Cauchy--Schwarz and the last follows from \eqref{eq:prelim-bounding-diff-newtonian-potential}, with the bounded factor $\abs{\hat\omega}^2$ absorbed into $C$. For the second term in \eqref{eq:def-rest-term-W}, we use the following inequality for $z,w \in \C$:
    \begin{equation}\label{eq:helper-inequality-nonlinterm}
        \abs{\abs{z}^2z-\abs{w}^2w}
        \leq \frac{3}{2}(\abs{z}^2+\abs{w}^2)\abs{z-w},
    \end{equation}
    which follows directly from algebra and the triangle inequality. For $x\in D$, set $z:=a\phi_j(x)+W(x)$ and $w:=a\phi_j(x)$. These pointwise evaluations are well defined because $H^2(D)\hookrightarrow L^\infty(D)$ in three dimensions. We then obtain
    \begin{equation}
     \begin{aligned}   \norm{\abs{a\phi_j+W}^2(a\phi_j+W) - \abs{a}^2 a \abs{\phi_j}^2 \phi_j}_{L^2(D)} &\leq \frac{3}{2} \norm{\big( \abs{a\phi_j + W}^2 + \abs{a\phi_j}^2 \big) \abs{a\phi_j + W - a\phi_j}}_{L^2(D)} \\
     & \leq \frac{3}{2} \big( \norm{\abs{a\phi_j +W}^2 + \abs{a\phi_j}^2}_{L^\infty(D)} \big) \norm{W}_{L^2(D)} \\
     & \leq \frac{3}{2}C (\abs{a} + \norm{W}_{H^2(D)})^2 \norm{W}_{H^2(D)},
     \end{aligned}
    \end{equation}
    for some constant $C>0$. The first inequality follows from \eqref{eq:helper-inequality-nonlinterm}; the remaining estimates use the Sobolev embedding and the fact that $\phi_j\in H^2(D)$. Since the eigenpair $(\lambda_j,\phi_j)$ is fixed throughout, its $H^2(D)$ norm is absorbed into $C$. Combining this estimate with the uniform bound on $\hat\omega^2\K_D^{\epsilon\hat\omega}$ and expanding $(\abs{a}+\norm{W}_{H^2(D)})^2\norm{W}_{H^2(D)}$ yields the claimed bound on $\mathcal{R}_W$.
\end{proof}

Neglecting the remainder gives the following leading-order amplitude equation:
\begin{equation}\label{eq:leading-order-amplitude-eq}
    \left(1-\hat\omega^2 \innerproduct{\phi_j}{\K_D^{\epsilon\hat\omega}[\phi_j]}_D
    -\eta\hat\omega^2 \abs{a}^2
    \innerproduct{\phi_j}{\K_D^{\epsilon\hat\omega}[\abs{\phi_j}^2\phi_j]}_D\right)a
    \approx \innerproduct{\phi_j}{u^{\mathrm{inc}}}_D.
\end{equation}
In the linear case $\eta=0$, we recover the same asymptotic behaviour of the amplitude coefficient as in \cite{ammari_linear_helmholtz_2019}. To analyze the implications of \eqref{eq:leading-order-amplitude-eq}, let $I := \abs{a}^2$ denote the intensity and set
\begin{equation}\label{eq:notation-amplitude-equation}
    K:=1-\hat\omega^2\innerproduct{\phi_j}{\K_D^{\epsilon\hat\omega}[\phi_j]}_D,\qquad
    M:=\hat\omega^2\innerproduct{\phi_j}{\K_D^{\epsilon\hat\omega}[\abs{\phi_j}^2\phi_j]}_D,\qquad
    f:=\innerproduct{\phi_j}{u^{\mathrm{inc}}}_D.
\end{equation}
After neglecting $\mathcal{R}_W$, equation \eqref{eq:leading-order-reduced-equation} becomes
\begin{equation}\label{eq:simplified-amplitude-relation}
    (K - \eta M \abs{a}^2)a = f.
\end{equation}
Taking the squared absolute value gives
\begin{equation}\label{eq:P-intensity-relation}
    P(I):=I \abs{K - \eta M I}^2 = \abs{f}^2.
\end{equation}
If $\abs{f}>0$ and $I>0$ solves \eqref{eq:P-intensity-relation}, then $K-\eta MI\neq0$ and
$a=f/(K-\eta MI)$ solves \eqref{eq:simplified-amplitude-relation}. Hence, for $\abs{f}>0$, the positive roots of \eqref{eq:P-intensity-relation} are in one-to-one correspondence with the complex solutions of \eqref{eq:simplified-amplitude-relation}.

Expanding $P$ gives
\begin{equation}
    P(I) = \abs{\eta M}^2 I^3 - 2 \Re\big(K\overline{\eta M}\big)I^2 + \abs{K}^2 I.
\end{equation}
The derivative $P'$ has two distinct positive roots if and only if
\begin{equation}\label{eq:conditions-distinct-roots}
    \Re\big(K\overline{\eta M}\big) > 0,
    \qquad
    4 \Re\big(K\overline{\eta M}\big)^2 > 3 \abs{\eta M}^2 \abs{K}^2.
\end{equation}
In that case, the critical points are
\begin{equation}
    I_{\pm} =
    \frac{2 \Re\big(K\overline{\eta M}\big)
    \pm \sqrt{4 \Re\big(K\overline{\eta M}\big)^2
    - 3 \abs{\eta M}^2\abs{K}^2}}
    {3 \abs{\eta M}^2},
\end{equation}
where $I_-$ is a local maximum and $I_+$ is a local minimum of $P$. Thus, for fixed $(\hat\omega,\epsilon,\eta,u^{\mathrm{inc}})$ with $\abs{f}>0$, equation \eqref{eq:P-intensity-relation} has three positive roots precisely when
\begin{equation}
    P(I_+)<\abs{f}^2<P(I_-).
\end{equation}
At either endpoint, one root has multiplicity two. There is exactly one positive root when $\abs{f}^2<P(I_+)$ or $\abs{f}^2>P(I_-)$. Consequently, the leading-order input--output relation between $f$ and $a$ can be multivalued.

Finally, we derive the corresponding representation of the scattered field. In the local regime of Proposition \ref{prop:local-reduction}, substitute $u=a\phi_j+W$ into \eqref{eq:scattered-field-equation} and define
\begin{equation}\label{eq:def-exterior-terms}
\begin{aligned}
    &\Phi_j^{\epsilon\hat\omega}(x) := \hat\omega^2\int_D {G^{\epsilon\hat\omega}(x-y)\phi_j(y)\dd y}, \qquad x \in \R^3 \setminus \overline{D}, \\
    &\mathcal{R}_W^{\mathrm{sc}}(x) := \hat\omega^2 \int_D G^{\epsilon\hat\omega}(x-y)
    \big[W(y) + \eta \abs{a\phi_j(y) + W(y)}^2(a\phi_j(y) + W(y))\big]\dd y,
    \qquad x \in \R^3 \setminus \overline{D}.
    \end{aligned}
\end{equation}
Then
\begin{equation}\label{eq:representation-of-u-scattered}
    \begin{aligned}
    u^{\mathrm{sc}}(x)
    ={}&a \Phi_j^{\epsilon\hat\omega}(x)
    + \mathcal{R}_W^{\mathrm{sc}}(x), \qquad x \in \R^3 \setminus \overline{D}.
    \end{aligned}
\end{equation}
The remainder includes a nonlinear term that does not contain $W$. The following proposition isolates the $O(\abs{a})$ resonant contribution and bounds the remainder in the local regime. It is one of our main contributions in this paper. 
\begin{proposition}\label{prop:scattered-final-prop}
    Let the assumptions and constraints of Proposition \ref{prop:local-reduction} hold, and consider any solution $u=a\phi_j+W_\eta(a,\hat\omega,\epsilon,u^{\mathrm{inc}})$ satisfying $G_j(a,\hat\omega,\epsilon,\eta,u^{\mathrm{inc}})=0$. Then, for every bounded domain $\Omega\Subset\R^3\setminus\overline{D}$, there exists a constant $C_\Omega>0$ such that
    \begin{equation}
        \norm{u^{\mathrm{sc}} - a\Phi_j^{\epsilon\hat\omega}}_{H^2(\Omega)}
        \leq C_{\Omega}\big(\norm{Q_ju^{\mathrm{inc}}}_{H^2(D)} + \epsilon \abs{a} + \eta\abs{a}^3\big).
    \end{equation}
\end{proposition}
\begin{proof}
    Since $\Omega$ is compactly contained in $\R^3\setminus\overline{D}$, there exists a constant $C(\Omega)>0$ such that
    \begin{equation}\label{eq:rest-scattered-bound}
        \norm{\mathcal{R}_W^{\mathrm{sc}}}_{H^2(\Omega)}
        \leq C(\Omega)\norm{W+\eta\abs{a\phi_j+W}^2(a\phi_j+W)}_{L^2(D)}.
    \end{equation}
    Applying standard estimates to the right-hand side of \eqref{eq:rest-scattered-bound}, together with Lemma \ref{lem:uniform-W-bound}, gives
    \begin{equation}
    \begin{aligned}
        C(\Omega) \norm{W+\eta\abs{a\phi_j+W}^2(a\phi_j+W)}_{L^2(D)} &\leq C(\Omega) \big( \norm{W}_{H^2(D)} + \eta( \abs{a} + \norm{W}_{H^2(D)} )^3 \big)  \\
        &\leq C(\Omega)\big( \norm{Q_j u^{\mathrm{inc}}}_{H^2(D)} + \epsilon \abs{a} + \eta \abs{a}^3 \\ 
        &+ \eta(\abs{a} + \norm{Q_j u^{\mathrm{inc}}}_{H^2(D)} + \epsilon \abs{a} + \eta \abs{a}^3)^3 \big) \\
        & \leq C(\Omega) (\norm{Q_j u^{\mathrm{inc}}}_{H^2(D)} + \epsilon \abs{a} + \eta \abs{a}^3).
    \end{aligned}
    \end{equation}
    In the second inequality, the constant from Lemma \ref{lem:uniform-W-bound} is absorbed into $C(\Omega)$. The final inequality follows from Lemma \ref{lem:helper-W-bound}, after decreasing the local radius if necessary and absorbing the resulting constants into $C(\Omega)$.
\end{proof}
\begin{remark}
    The conclusion of Proposition \ref{prop:scattered-final-prop} does not generally extend to the second regime of Proposition \ref{prop:alternative-regime-local-reduction}, where $\norm{u^{\mathrm{inc}}}_{H^2(D)}=O(1)$ and $\eta\to0$. In that regime, $W_\eta=h_{\mathrm{lin}}+O(\eta)$, so the appropriate leading-order approximation is $u=a\phi_j+h_{\mathrm{lin}}+O(\eta)$. By contrast, Proposition \ref{prop:scattered-final-prop} controls the remainder by $\norm{Q_ju^{\mathrm{inc}}}_{H^2(D)}+\epsilon\abs{a}+\eta\abs{a}^3$.
\end{remark}

\section{Numerical Experiments}\label{sec:numerical}

\subsection{Nonlinear Solver}\label{sec:nonlin-solver}

For simplicity of presentation, consider the spherical inclusion $D = B(0,R)$. We aim to solve the nonlinear Lippmann--Schwinger equation without an incident field, i.e., with $u^{\mathrm{inc}}=0$, by applying an iterative Newton method initialized at a normalized linear resonant pair. We begin with the nonlinear equation \cite[Eq. 3.2]{ammari_dielectric_2025}:
\begin{equation}\label{eq:num-nonlin-LS}
  u = \tau \,\omega^2 \K_D^\omega[N(u)],
  \qquad
  N(u) = u + \eta\abs{u}^2 u,
\end{equation}
where $\K_D^\omega$ denotes the volume integral operator from \eqref{eq:prelim-K-volume-operator-def} and $G^\omega(x)$ denotes the outgoing Green function of $-(\Delta + \omega^2)$ in $\R^3$, given by \eqref{Gomega} for $\omega \in \C$. Recall also that $G^\omega$ admits the following expansion about $\omega=0$:
\begin{equation}
    G^\omega(x) = \sum_{n=0}^{\infty}{\omega^n G_n(x)}, \quad \mathrm{with} \ G_n(x):= \frac{\iu^n \abs{x}^{n-1}}{4 \pi n!}, \qquad x\neq0.
\end{equation}

To select a solution on a branch with fixed normalization $\mathcal{N}>0$, we impose a normalization constraint. Since the equation is $S^1$-equivariant, the linearized system retains a phase null direction unless a local phase condition is also imposed. We therefore use
\begin{equation}
  \|u\|_W^2 := \sum_{j=1}^n |u_j|^2 W_j = \mathcal{N},
  \qquad
  \Im \langle u,u_{\mathrm{lin}} \rangle_W = 0,
\end{equation}
where $W_j$ are the quadrature weights, $u_{\mathrm{lin}}$ is a linear reference mode, and $\innerproduct{u}{v}_W := \sum_{j=1}^n{\overline{u_j}v_j W_j}$ is a weighted inner product, so that $\norm{u}^2_W = \innerproduct{u}{u}_W$. The second condition fixes the phase locally along the seeded continuation branch.
The method extends to arbitrarily shaped inclusions, as illustrated in Section \ref{sec:single-resonator}. For the present discussion, consider a tensor-product quadrature of the ball $B(0,R)$ in $(r,\theta,\phi)$, with node coordinates $X_j \in \R^3$, weights $W_j>0$, and equivalent-cell radii
\begin{equation}
    R_{\mathrm{eq},j} = \left(\frac{3W_j}{4\pi}\right)^{1/3},
    \qquad j=1,\ldots,n,
\end{equation}
for diagonal regularization. We use Gauss--Legendre quadrature in $r \in [0,R]$ and $s = \cos\theta \in [-1,1]$, together with the trapezoidal rule in $\phi \in [0,2\pi)$. Since $\dd V = r^2\,\dd r\,\dd s\,\dd\phi$, the resulting nodes and weights satisfy
\begin{equation}
    \int_{B(0,R)} f(x)\,\dd x \approx \sum_{j=1}^n f(X_j)\,W_j.
\end{equation}

The discrete kernel matrix $K_h^\omega \in\C^{n\times n}$ corresponding to the Green function has the form 
\begin{equation}\label{num:discrete-operator}
  (K_h^\omega)_{ij} =
  \begin{cases}
    \dfrac{e^{\iu \omega |X_i-X_j|}}{4\pi |X_i-X_j|} W_j, & i\neq j,\\[4pt]
    \displaystyle\int_{B(0,R_{\mathrm{eq},i})}\frac{e^{\iu \omega r}}{4\pi r}\,\dd V,
      & i=j,
  \end{cases}
\end{equation}
where the diagonal entry is the self-interaction term
\begin{equation}\label{eq:self-interaction-discrete}
  \int_{\abs{\rho} < R_{\mathrm{eq},i}}G^\omega(\rho)\dd \rho
  = \int_0^{R_{\mathrm{eq},i}} e^{\iu\omega r}\,r\,\dd r.
\end{equation}
This diagonal integral has the closed form
\begin{equation}\label{eq:num-closed-form-diag}
    \int_{0}^{R_{\mathrm{eq},i}} e^{\iu \omega r} r \dd r
    = \frac{R_{\mathrm{eq},i} e^{\iu \omega R_{\mathrm{eq},i}}}{\iu \omega}
    + \frac{e^{\iu \omega R_{\mathrm{eq},i}}-1}{\omega^2}.
\end{equation}
For the Jacobian, we also need the derivative of this discrete kernel with respect to $\omega$. The off-diagonal terms satisfy $\partial_\omega (K_h^\omega)_{ij}  = \frac{\iu e^{\iu \omega \abs{X_i - X_j}}}{4 \pi} W_j$ for $i \neq j$. Although the derivative of a diagonal term follows directly from \eqref{eq:num-closed-form-diag}, the implementation uses its Taylor expansion when $\abs{\omega R_{\mathrm{eq},i}} < 10^{-3}$:
\begin{equation}
    \frac{\dd}{\dd \omega} \int_{0}^{R_{\mathrm{eq},i}} e^{\iu \omega r}r \dd r
    = \frac{\iu R_{\mathrm{eq},i}^3}{3}
    - \frac{\omega R_{\mathrm{eq},i}^4}{4}
    - \frac{\iu \omega^2 R_{\mathrm{eq},i}^5}{10}
    + O(\abs{\omega}^3 R_{\mathrm{eq},i}^6).
\end{equation}

Theorem 3.5 of \cite{ammari_dielectric_2025} indicates that nonlinear dielectric resonances are \emph{close} to linear ones. This motivates a numerical implementation based on Newton's method, with the initial seed given by a linear dielectric resonant pair $(\omega_0, u_0)$. Recall that in the linear case, when $\eta=0$, we have $q_\eta(u) = 1$, and therefore the Lippmann--Schwinger equation \eqref{eq:num-nonlin-LS} reduces to
\begin{equation}\label{num:eq-linear-LS}
    u = \tau \omega^2 \K_D^\omega [u].
\end{equation}
Working in the scaled variables $\hat\omega:=\epsilon^{-1}\omega$ and $\epsilon:=\tau^{-1/2}$, we know from \cite{ammari_maxwell_2023} that, as $\tau \rightarrow \infty$, the problem reduces to the following linear eigenvalue problem:
\begin{equation}\label{eq:num-static-linear-LS}
    u = \hat\omega^2 \K_D[u].
\end{equation}
The Gohberg--Sigal theory \cite{gohberg-sigal} then gives the following asymptotics.
\begin{proposition}\label{prop:num-linear-asymp}
    Let $d=3$. When the contrast $\tau$ is sufficiently large, the scattering resonances for \eqref{num:eq-linear-LS} exist in the subwavelength regime, with the following asymptotic behaviour:
    \begin{equation*}
        \omega(\tau) = \dfrac{1}{\sqrt{\tau \lambda_j}} + O(\tau^{-1}), \quad \mathrm{as} \ \tau \rightarrow \infty,
    \end{equation*}
    with the associated normalized resonant state
    \begin{equation*}
        u = \phi_j + O(\tau^{-1/2}), \quad \norm{\phi_j}_{L^2(D)} = 1,
    \end{equation*}
    where $(\lambda_j, \phi_j)$ is an eigenpair of the Newtonian potential $\K_D$, i.e., $\K_D[\phi_j] = \lambda_j \phi_j$.
\end{proposition}

Let $(\lambda_0,\phi_0)$ denote the principal eigenpair of the discrete counterpart of the Newtonian potential, normalized by $\norm{\phi_0}_W^2=\sum_{j=1}^{n}\abs{\phi_{0,j}}^2W_j=1$. The asymptotics in Proposition \ref{prop:num-linear-asymp} suggest the seed $\omega_{\mathrm{lin}}\approx(\tau\lambda_0)^{-1/2}$. For robustness, we add a small negative imaginary part so that Newton's method starts in the lower half-plane $\Im\omega<0$ associated with dielectric resonances.

Consider now the discrete analogue of \eqref{eq:num-nonlin-LS}:
\begin{equation}\label{eq:discrete-nonlinear-problem-homog}
    F_h(u, \omega) := u - \tau \omega^2 K_h^\omega[N(u)] = 0,
\end{equation}
where the nonlinearity acts componentwise and, for a single resonator, the unknown is the pair $(u, \omega) \in \C^n \times \C$. To obtain a well-posed problem for the implementation, we additionally impose the normalization and phase constraints, resulting in the system
\begin{equation}\label{num:system-nonlin-solver}
  M_h(u,\omega) := \begin{pmatrix}
    F_h(u,\omega)\\
    \|u\|_W^2 - \mathcal{N}\\
    \Im\langle u,u_{\mathrm{lin}}\rangle_W
  \end{pmatrix} = 0,
\end{equation}
where $u_{\mathrm{lin}}$ is the normalized discrete linear mode described above.
\begin{remark}
    It is important to recall that the nonlinearity $N(u)$ is only real Fr\'echet differentiable. Thus, the complex unknowns are split into real and imaginary parts, and Newton's method is applied to this ``doubled'' system of real variables.
\end{remark}
Concretely, let $\vb{x} = (\Re u, \Im u, \Re \omega, \Im \omega) \in \R^{2n+2}$, where $n \in \N_+$ is the total number of quadrature points. The Jacobian of \eqref{num:system-nonlin-solver}, denoted by $J_h(\vb x) := D_{\vb x} M_h(\vb x)
  \in \R^{(2n+2)\times(2n+2)}$, has the following block matrix form:
\begin{equation}\label{num:jac-block-matrix}
    J_h(\vb{x}) =
    \begin{pmatrix}
        J_u & J_\omega \\
        J_{\mathrm{normalization}} & \vb{0}_{1\times 2} \\
        J_{\mathrm{phase}} &  \vb{0}_{1\times 2}
    \end{pmatrix},
\end{equation}
where the terms $J_{\mathrm{normalization}}$ and $J_{\mathrm{phase}}$ are the differentials corresponding to the normalization and phase constraints from \eqref{num:system-nonlin-solver}. Using \cite[Eq. 3.7]{ammari_dielectric_2025}, we obtain, in the discrete case,
\begin{equation}\label{eq:num-disc-F-deriv}
    D_uF_h(u,\omega)[v]
    = v - \tau \omega^2 K_h^\omega\big[ (1+2\eta\abs{u}^2)v
    + \eta u^2 \overline{v} \big], \quad v\in \C^n.
\end{equation}
For $u \in \C^n$, split the derivative into two matrix terms:
\begin{equation}\begin{aligned}
    A &:= I - \tau\omega^2 K_h^\omega\mathrm{diag}(1+2\eta\abs{u}^2),\\
    B &:= -\tau\omega^2K_h^\omega \mathrm{diag}(\eta u^2),
\end{aligned}
\end{equation}
with $\mathrm{diag}(v)$ denoting the $\C^{n\times n}$ diagonal matrix formed by the entries of $v \in \C^n$. This notation allows us to write \eqref{eq:num-disc-F-deriv} as $D_uF_h(u,\omega)[v] = Av + B\overline{v}$, and then the first block term $J_u \in \R^{2n\times 2n}$ is the unique matrix satisfying
\begin{equation}
    \begin{pmatrix}
        \Re(D_uF_h(u,\omega)[v]) \\
        \Im(D_uF_h(u,\omega)[v])
    \end{pmatrix}
    = J_u \begin{pmatrix}
        \Re(v) \\
        \Im(v)
    \end{pmatrix}, \quad \mathrm{for} \ v\in \C^n,
\end{equation}
where the previous notation results in the following expression for this block:
\begin{equation}
    J_u = \begin{pmatrix}
        \Re (A+B) & -\Im (A-B) \\
        \Im (A+B) & \Re (A-B) 
    \end{pmatrix}.
\end{equation}

The remaining block in the first row, $J_\omega$, can be described similarly. The expansion \eqref{eq:prelim-K-volume-operator-expansion} and the complex Fr\'echet differentiability of the map $\omega \mapsto F_h(u,\omega)$ give, for $z\in\C$,
\begin{equation}
    D_\omega F_h(u,\omega)[z]
    = -\tau \big(2\omega K_h^\omega +\omega^2 \partial_\omega K_h^\omega\big)N(u)\,z.
\end{equation}
It follows that $J_\omega \in \R^{2n\times 2}$ is the real matrix satisfying
\begin{equation}
    \begin{pmatrix}
        \Re(D_\omega F_h(u,\omega)[z]) \\
        \Im(D_\omega F_h(u,\omega)[z])
    \end{pmatrix}=
    J_{\omega} \begin{pmatrix}
        \Re(z) \\
        \Im(z)
    \end{pmatrix}, \quad \mathrm{for} \ z \in \C.
\end{equation}
Set $f_\omega := D_\omega F_h(u,\omega)[1] \in \C^n$. Since the derivative is complex linear, $D_\omega F_h(u,\omega)[z] = f_\omega z$, and hence
\begin{equation}
    J_\omega = \begin{pmatrix}
        \Re f_\omega & -\Im f_\omega \\
        \Im f_\omega & \Re f_\omega
    \end{pmatrix}.
\end{equation}

At iteration $\ell$, let $\delta\vb{x}^{(\ell)}$ denote the Newton correction determined by
\begin{equation}\label{num:iterative-newton}
\begin{aligned}
    &J_h(\vb{x}^{(\ell)})\delta\vb{x}^{(\ell)}=-M_h(\vb{x}^{(\ell)}), \\
    &\vb{x}^{(\ell+1)}=\vb{x}^{(\ell)}+\alpha_\ell\delta\vb{x}^{(\ell)},
\end{aligned}
\end{equation}
where the implementation uses a backtracking line search. The full Newton step $\alpha_\ell=1$ is tried first; if it is infeasible, non-finite, or fails to reduce the residual norm, the step size is halved. The condition $\Im\omega<0$ is enforced by rejecting steps that leave the lower half-plane.

For a resonator dimer $D=D_1\cup D_2$, the discretized problem \eqref{eq:discrete-nonlinear-problem-homog} takes a block form. The state is $u=(u_1,u_2)\in\C^{2n}$, and the discrete integral operator is
\begin{equation}\label{eq:num-blockmat-dimer}
    K_h^{\omega, L} = \begin{pmatrix}
        K_{\mathrm{self}}^\omega & K_{12}^{\omega, L} \\
        K_{21}^{\omega, L} & K_{\mathrm{self}}^\omega
    \end{pmatrix},
\end{equation}
where the diagonal self-interaction terms are given by \eqref{num:discrete-operator}. The cross-interaction terms between the resonators have the form
\begin{equation}
    (K_{12}^{\omega, L})_{ij} = \dfrac{e^{\iu \omega r_{ij}^{12}(L)}}{4 \pi r_{ij}^{12}(L)}W_j,
\end{equation}
with $K_{21}^{\omega,L}$ defined analogously. Given reference quadrature nodes $X_i=(x_i,y_i,z_i)$ on a single resonator, we denote the cross-interaction distance by $r_{ij}^{12}(L):=\abs{X_i^{(1)}-X_j^{(2)}}$. The superscripts indicate that the quadrature nodes are translated to the first and second resonators, respectively.

Accordingly, the discrete dimer problem is
\begin{equation}
    F_h(u_1,u_2,\omega) = \begin{pmatrix}
        u_1 \\
        u_2
    \end{pmatrix} - \tau \omega^2 K_h^{\omega, L} \begin{pmatrix}
        N(u_1) \\
        N(u_2)
    \end{pmatrix} = 0,
\end{equation}
where $\norm{u_1}^2_W+\norm{u_2}^2_W=\mathcal{N}$ and the unknown is represented by the real vector
\begin{equation}
    \vb{x} = (\Re u_1,\Re u_2, \Im u_1,\Im u_2,\Re \omega,\Im \omega) \in \R^{4n+2}. 
\end{equation}

To exhibit symmetry breaking from the symmetric branch and track the asymmetric branches, we employ continuation methods of the kind available in bifurcation packages such as the Julia package \lstinline{BifurcationKit} \cite{veltz:hal-02902346}. The current Python implementation supports natural Newton continuation, deflated continuation \cite{farrel-2015}, and pseudo-arclength continuation (PALC).

For clarity, we describe the PALC method used to generate the bifurcation diagrams in Section \ref{sec:dimer}. This method, presented in \cite{Keller1988LecturesON,rabinowitz}, seeks solution curves of $F(x,p)=0$, where $p$ is a real parameter. Under the usual regularity assumptions, a known solution $(x_0,p_0)$ lies on a local solution curve
\begin{equation}
    \gamma(s)=(x(s),p(s)), \qquad s\in I.
\end{equation}

The PALC method \cite{Keller1988LecturesON} augments the nonlinear system with a hyperplane condition. Let $m$ denote the number of real unknowns in $x$, and define
\begin{equation}
    \innerproduct{(v,q)}{(w,r)}_\theta
    := \frac{2-\theta}{m}v^\top w+\theta qr,
    \qquad \theta\in(0,1].
\end{equation}
If $t_0=(t_{x,0},t_{p,0})$ is the unit tangent at $(x_0,p_0)$, the hyperplane condition is
\begin{equation}\label{eq:palc-constraint}
    \mathcal{C}(x,p)
    := \innerproduct{(x-x_0,p-p_0)}{t_0}_\theta-\Delta s=0.
\end{equation}
Here, $\theta$ controls the relative weights of the state and parameter components, and $\Delta s$ is the pseudo-arclength step. In our setting, consider the map $H:\R^m\times\R\to\R^m$,
\begin{equation}\label{eq:palc-map}
    H(x,p) := \begin{pmatrix}
    \Re (F_h(u,\omega;p)) \\
    \Im (F_h(u,\omega;p))\\
    \|u\|_W^2 - \mathcal{N}\\
    \Im\langle u,u_{\mathrm{lin}}\rangle_W
  \end{pmatrix} = 0,
\end{equation}
which is the real form of \eqref{num:system-nonlin-solver}. Here $m=2n+2$ for a single resonator and $m=4n+2$ for a dimer. One may take either $p=\delta$, the perturbation of the linear coefficient in \eqref{eq:linear-imperfection}, or $p=\mathcal{N}$, replacing $\mathcal{N}$ by $p$ in the normalization row of \eqref{eq:palc-map}.

Let $H_x$ and $H_p$ denote the partial derivatives of $H$. At a regular point,
\begin{equation}
    \operatorname{rank}(H_x\ \ H_p)=m.
\end{equation}
Thus, $DH=(H_x\ \ H_p)$ has a one-dimensional kernel, and $H^{-1}(0)$ is locally a smooth curve \cite[Ch. 5]{Keller1988LecturesON}. Differentiating $H(x(s),p(s))=0$ gives
\begin{equation}
    H_x(x(s),p(s))x'(s)+H_p(x(s),p(s))p'(s)=0,
    \qquad
    (x'(s),p'(s))\in\ker DH.
\end{equation}
For the reduced dimer problem, the corresponding local asymmetric branch is obtained from the invertibility of
\begin{equation}
    \frac{\partial(\Re F_+,\Re F_-,\Im F_+,\Im F_-)}
    {\partial(\Re\hat\omega,\Im\hat\omega,p_+,\Delta\theta)} \ ,
\end{equation}
at $(\Re\hat\omega,\Im\hat\omega,p_+,\Delta\theta)
=(\hat\omega_*,0,p_{+,*},0)$, where $\hat\omega_*$ and $p_{+,*}$ denote the
limiting bifurcation frequency and symmetric-mode amplitude, respectively; see
\cite[Sec. 4.1]{ammari_dielectric_2025}.

At a solution point $(x_k,p_k)$, compute a nonzero null vector of
$DH(x_k,p_k)$, for example by a bordered linear solve, and normalize it with respect to
$\norm{\cdot}_\theta$. The unit tangent is
\begin{equation}\label{eq:unit-tangent-palc}
    t_k := (t_{x,k}, t_{p,k}) \in \R^m \times \R.
\end{equation}
Its sign is chosen so that $\innerproduct{t_k}{t_{k-1}}_\theta>0$. This nullspace construction remains valid at folds, where $H_x$ may be singular. The predictor is $(x^0,p^0)=(x_k,p_k)+t_k\Delta s$, and the corrector solves
\begin{equation}\label{eq:palc-full-system}
    \begin{pmatrix}
        H(x,p) \\
        \innerproduct{t_k}{(x-x_k, p-p_k)}_\theta - \Delta s
    \end{pmatrix} = 0,
\end{equation}
which is solved using Newton's method with a backtracking line search and an adaptively adjusted step size $\Delta s$.
\begin{remark}
    In the mirror-symmetric dimer computations, we use standard techniques from bifurcation theory \cite[Sec. 5.26]{Keller1988LecturesON} to switch to the asymmetric branches. We first break the $\Z_2$-symmetry with the linear-coefficient perturbation $\delta>0$ and use PALC with $p=\mathcal{N}$ to find a perturbed solution at some $\mathcal{N}_p>\mathcal{N}_{\mathrm{crit}}$. Starting from that solution, natural Newton continuation in $p=\delta$ yields the unperturbed solution at $\delta=0$ and normalization $\mathcal{N}_p$, beyond the bifurcation point. Finally, PALC in $p=\mathcal{N}$ traces the asymmetric branch back to its bifurcation from the symmetric branch.
\end{remark}

We finally note that the numerical framework extends to the problem with an incident wave in \eqref{eq:nonlinear-helmholtz}. The discretization remains unchanged, and the discrete nonlinear equation is
\begin{equation}\label{eq:discrete-equation-incfield}
    u - u^{\mathrm{inc}} - \tau\omega^2 K_h^{\omega}[N(u)] = 0.
\end{equation}
This is the incident-wave analogue of \eqref{eq:discrete-nonlinear-problem-homog}, with the frequency $\omega$ prescribed rather than unknown. The problem therefore lacks the $S^1$ phase symmetry of the homogeneous eigenvalue problem and requires neither the phase nor the normalization constraint. It yields a square real system for
\begin{equation}
    \vb{x}=(\Re u,\Im u)\in\R^{2n}.
\end{equation}

\subsection{Single Resonator}\label{sec:single-resonator}

In this section, we focus on a single dielectric resonator $D\subset \R^3$, which may be nonspherical, and study the asymptotic behaviour of the nonlinear dielectric resonant pair $(u,\omega)$. The implementation is written for a resonator dimer, whose kernel matrix has the block form \eqref{eq:num-blockmat-dimer}. For a single resonator, the cross-interaction blocks are simply omitted, leaving the matrix in \eqref{num:discrete-operator}.

For the asymptotic discussion, fix a sufficiently small normalization constant $\mathcal{N}>0$ and a sufficiently large contrast $\tau$. Let $\omega(\mathcal{N},\tau)$ be the nonlinear dielectric resonance solving \eqref{eq:num-nonlin-LS}, whose local existence in this regime follows from \cite[Thm. 3.5]{ammari_dielectric_2025}. The numerical continuation below also explores values outside this proved local regime. Throughout, $\hat\omega:=\epsilon^{-1}\omega$ denotes the scaled frequency, with $\epsilon:=\tau^{-1/2}$. We also define $\hat\omega_j := \lambda_j^{-1/2}$, where $\lambda_j$ is a simple eigenvalue of the Newtonian potential $\K_D$ with an $L^2$-normalized eigenfunction $\phi_j$.

From \cite[Cor. 3.3]{ammari_dielectric_2025}, we know that the nonlinear dielectric resonances exist near the linear ones in the high-contrast, small-amplitude regime. Moreover, from \cite[Cor. 3.4]{ammari_dielectric_2025} we have the following asymptotic for the scaled nonlinear resonance:
\begin{equation}\label{eq:asympt-dielec-reso-nonlinear}
    \hat\omega(a,\epsilon) = \hat\omega_j - \frac{\iu \hat\omega_j^4}{8 \pi}\big( \int_{D}{\phi_j \dd x} \big)^2 \epsilon + O(\abs{a}^2 + \epsilon^2).
\end{equation}
Consequently, in the joint small-amplitude, high-contrast limit $(a,\epsilon)\to(0,0)$,
\begin{equation}\label{eq:scaling-law-tau}
    \begin{aligned}
    \Im \hat\omega(a,\epsilon)
    &= -c_j\epsilon + O(\abs{a}^2+\epsilon^2),\\
    \Re \hat\omega(a,\epsilon)
    &= \hat\omega_j + O(\abs{a}^2+\epsilon^2),
    \end{aligned}
    \qquad
    c_j:=\frac{\hat\omega_j^4}{8\pi}
    \left(\int_D\phi_j\dd x\right)^2.
\end{equation}
We choose $\phi_j$ to be real and normalized, so $c_j\geq0$. Equation \eqref{eq:scaling-law-tau} is a two-parameter expansion; at fixed nonzero $\mathcal{N}$, the plots below test the expected dominance of these leading terms rather than a separate one-parameter limit theorem.
Let $(\omega_*,\phi_*)$ denote the linear subwavelength dielectric resonant pair satisfying
\begin{equation}
    \begin{aligned}
    \omega_* &= \tau^{-1/2} \hat\omega_j + O(\tau^{-1}),\\
    \phi_* &= \phi_j + O(\tau^{-1/2}),
    \qquad \norm{\phi_j}^2_{L^2(D)}=1,
    \qquad \phi_* - \phi_j \perp \phi_j,
    \end{aligned}
\end{equation}
where $\phi_*$ is the linear dielectric resonant state corresponding to $(\lambda_j, \phi_j)$. From \cite[Remark 2]{ammari_dielectric_2025}, in the high-contrast, small-amplitude regime, the mapping $\abs{a} \mapsto \mathcal{N}$ is invertible, which allows us to reparametrize the nonlinear solution pair $(\omega(a,\epsilon), u(a,\epsilon))$ as $(\omega(\mathcal{N},\tau), u(\mathcal{N},\tau))$. Thus, we expect the following asymptotics for the \emph{unscaled} nonlinear dielectric resonance:
\begin{equation}\label{eq:scaling-linear-ON}
    \abs{\omega(\mathcal{N},\tau) - \omega_*(\tau)} = O(\mathcal{N}),
\end{equation}
for a fixed high contrast $\tau$.

Using the discretization described in Section \ref{sec:nonlin-solver}, we present numerical results for a single resonator with three choices of $D$: a sphere, an ellipsoid, and a peanut-shaped domain. We set the small normalization constant to $\mathcal{N}=0.01$ and plot $\Re\hat\omega$ and $\abs{\Im\hat\omega}$ for the resonance computed by Newton's method over the contrast range $50\leq\tau\leq25000$.
Figures \ref{fig:asymp_re_vs_contrast} and \ref{fig:asymp_im_vs_contrast} numerically exhibit the leading behaviour in \eqref{eq:scaling-law-tau}. In particular, the predicted slope of $\abs{\Im\hat\omega}$ as a function of $\epsilon$ is $c_j$. The different magnitudes reflect the geometry-dependent eigenpairs, including their volume dependence. Finally, we fix the high contrast $\tau=500$ and inspect the resonance as a function of $\mathcal{N}$, including values outside the small-amplitude regime. Figure \ref{fig:asymp_omega_vs_N} exhibits the approximate scaling in \eqref{eq:scaling-linear-ON}. The numerical values begin to deviate from the leading-order reference once $\mathcal{N}>1$ for the ellipsoid and sphere, and around $\mathcal{N}\approx0.5$ for the peanut shape.

\begin{figure}[htbp]%
    \centering
    \includegraphics[width=0.9\linewidth]{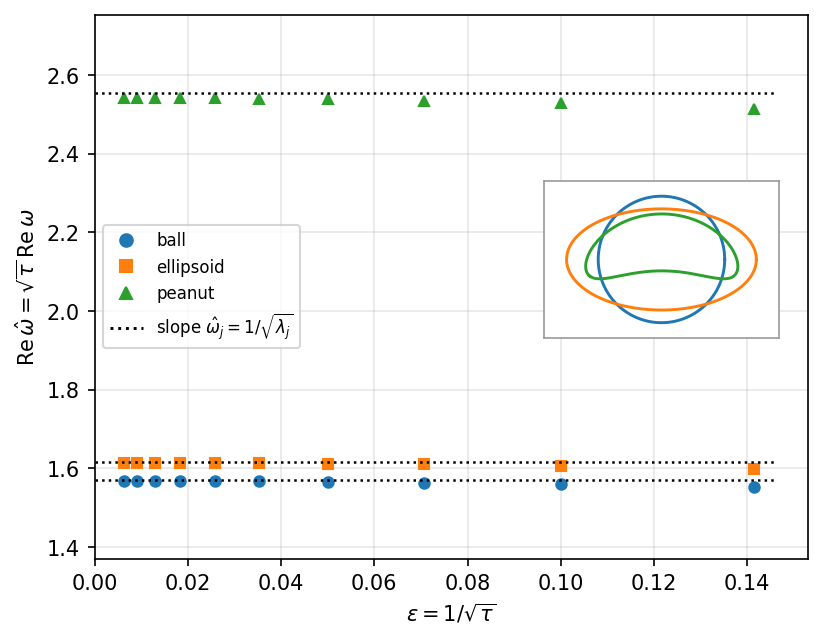}
    \caption{Real part of the scaled nonlinear dielectric resonance $\hat\omega(a,\epsilon)$ as a function of the contrast $\tau = 1/\epsilon^2$ for three resonator shapes. The leading eigenpair $(\lambda_j,\phi_j)$ is computed separately for each shape, so each reference line depends on the corresponding leading eigenpair of $\K_D$.}
    \label{fig:asymp_re_vs_contrast}
\end{figure}

\begin{figure}[htbp]%
    \centering
    \includegraphics[width=0.9\linewidth]{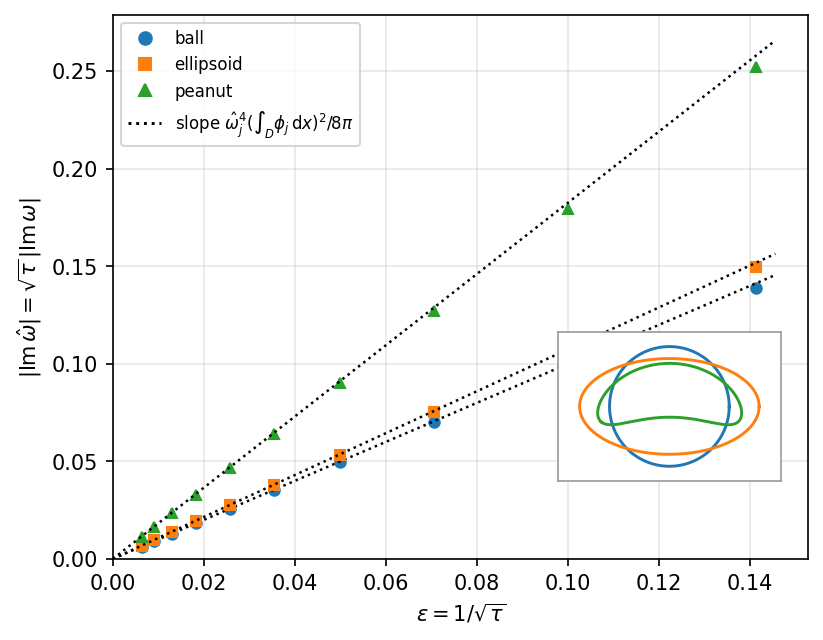}
    \caption{Magnitude of the imaginary part of the scaled nonlinear dielectric resonance $\hat\omega(a,\epsilon)$ as a function of the contrast $\tau = 1/\epsilon^2$ for three resonator shapes.}
    \label{fig:asymp_im_vs_contrast}
\end{figure}

\begin{figure}[htbp]%
    \centering
    \includegraphics[width=0.9\linewidth]{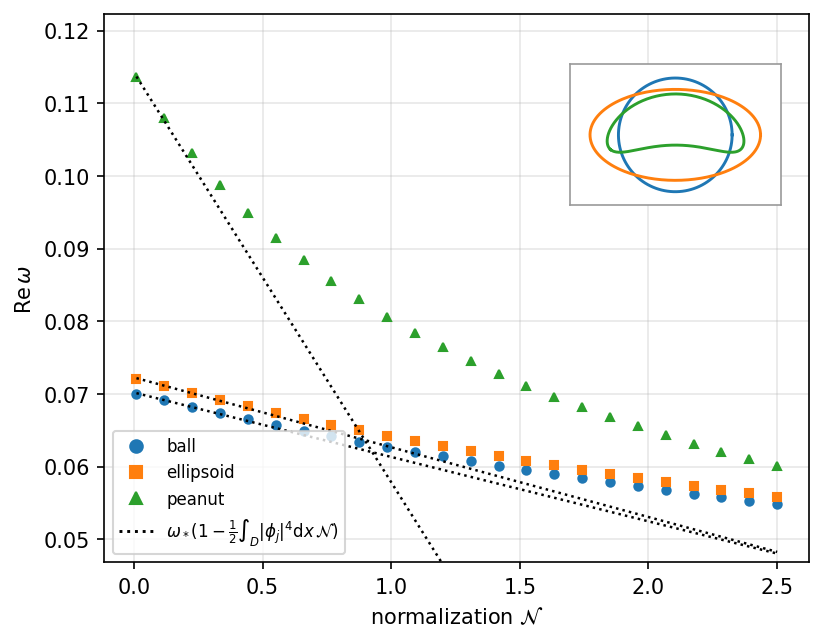}
    \caption{Real part of the nonlinear dielectric resonance as a function of the normalization constant $\mathcal{N}$ at the fixed high contrast $\tau=500$. Reference lines contain the leading-order $O(\mathcal{N})$ terms after reparametrization of equation \eqref{eq:asympt-dielec-reso-nonlinear}.}
    \label{fig:asymp_omega_vs_N}
\end{figure}

We next assess discretization convergence. For fixed contrast $\tau$ and normalization constant $\mathcal{N}$, we study whether the discrete nonlinear resonant state stabilizes under grid refinement and whether the discrete nonlinear dielectric frequency converges algebraically to a refined-grid reference.

Because the nonlinear problem is $S^1$-equivariant, the states must be phase-aligned before they are compared. Let $u_h^{\mathrm{rad}}(r)$ and $u_{\mathrm{ref}}^{\mathrm{rad}}(r)$ denote the shell-averaged radial profiles. We use the phase-invariant error
\begin{equation}\label{eq:radial-error}
    E_u(h) := \inf_{\vartheta\in[0,2\pi)}
    \left(
    \frac{\int_0^R\abs{e^{\iu\vartheta}u_h^{\mathrm{rad}}(r)
    -u_{\mathrm{ref}}^{\mathrm{rad}}(r)}^2r^2\dd r}
    {\int_0^R\abs{u_{\mathrm{ref}}^{\mathrm{rad}}(r)}^2r^2\dd r}
    \right)^{1/2}.
\end{equation}
Here, $u_{\mathrm{ref}}^{\mathrm{rad}}$ is computed on the finest grid. This does not prove the norm convergence of $u_h$, but it gives numerical evidence that the principal state stabilizes under mesh refinement.

For the right panel of Figure \ref{fig:exp6-single}, we compute the discrete nonlinear resonant frequency $\omega_h$ at the same fixed parameters. Let $N_{\mathrm{dof}}$ denote the number of quadrature points in the Nystr\"om discretization. The equivalent-cell radii satisfy $R_{\mathrm{eq},j}=(3W_j/(4\pi))^{1/3}$, and the Taylor expansion of the self-interaction in \eqref{eq:self-interaction-discrete} suggests a diagonal regularization error of order $O(R_{\mathrm{eq}}^2)$. Refinement from $(n_r,n_\theta,n_\phi)=(p,p,2p)$ to $(p+1,p+1,2(p+1))$ gives $R_{\mathrm{eq}}\sim N_{\mathrm{dof}}^{-1/3}$, suggesting the heuristic model
\begin{equation}
    \omega_h=\omega_\infty^{\mathrm{ext}}
    +cN_{\mathrm{dof}}^{-2/3}
    +o(N_{\mathrm{dof}}^{-2/3}).
\end{equation}
We therefore compare $\omega_h$ with the extrapolated intercept $\omega_\infty^{\mathrm{ext}}$.

Figure \ref{fig:exp6-single} shows $E_u(h)$ from \eqref{eq:radial-error} in the left panel and $\abs{\omega_h-\omega_\infty^{\mathrm{ext}}}$ in the right panel. Both plots support the preceding refinement heuristics.

\begin{figure}[htbp]
    \centering
    \includegraphics[width=0.9\linewidth]{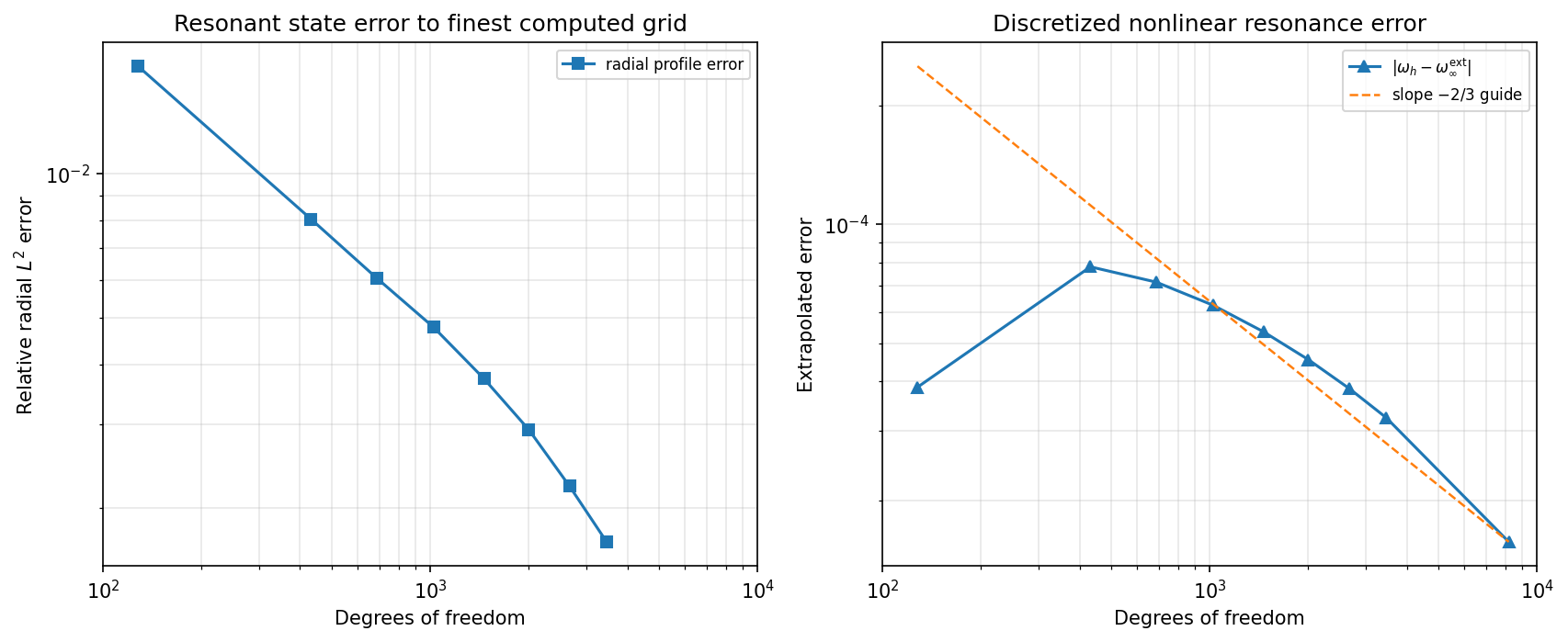}
    \caption{Grid-independence study. (Left) Phase-aligned radial error $E_u$. (Right) Nonlinear dielectric resonant-frequency error.}
    \label{fig:exp6-single}
\end{figure}

We next compare our numerical framework with the asymptotic analysis of nonlinear spherical scatterers in \cite{meklachi_asymptotic_2018}. The authors consider a spherical scatterer of radius $r_h=h$ governed by
\begin{equation}\label{eq:mms18}
    \Delta u_h+k_h^2(1+\eta_h)u_h
    +k_h^2\beta_h\abs{u_h}^2u_h=0
    \quad \mathrm{in}\ \R^3,
\end{equation}
with the outgoing radiation condition, where $B_h:=hB$, $\eta_h=h^{-2}\eta_0\chi_{B_h}$, and $\beta_h=h^{-2}\beta_0\chi_{B_h}$. Here $h$ denotes the particle-size parameter, not the quadrature index used elsewhere in this section. After translating the notation, we reproduce their setting by replacing the product $\tau N(u)$ in \eqref{eq:num-nonlin-LS} with $N_h(u):=\eta_hu+\beta_h\abs{u}^2u$.

Figure \ref{fig:mms18} reproduces the reported convergence of the size-dependent solution $u_h$ to the limiting field $u_0$ in \cite[Eq.~(53)]{meklachi_asymptotic_2018}.

\begin{figure}%
    \centering
    \begin{subfigure}[b]{0.45\textwidth}
        \centering
        \includegraphics[width=\linewidth]{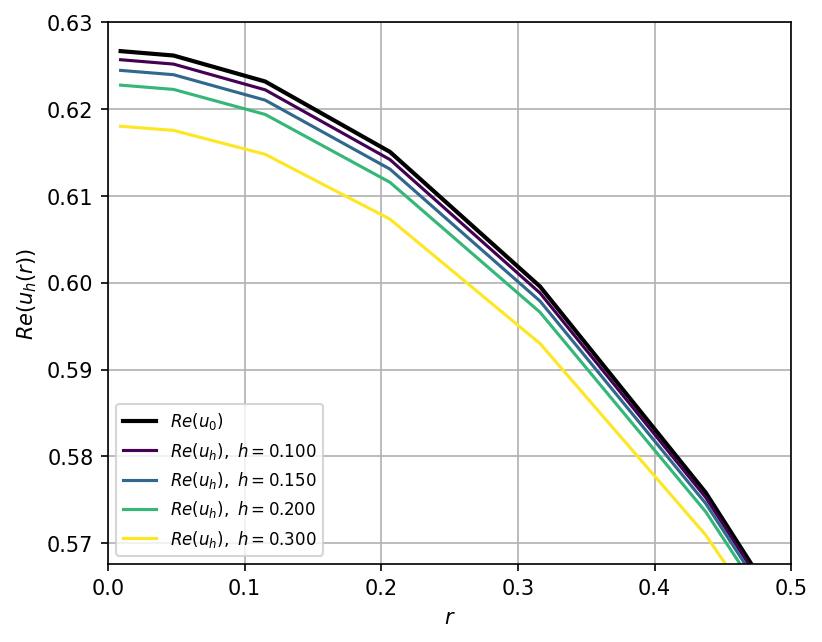}
    \end{subfigure}
    \begin{subfigure}[b]{0.45\textwidth}
        \centering
        \includegraphics[width=\linewidth]{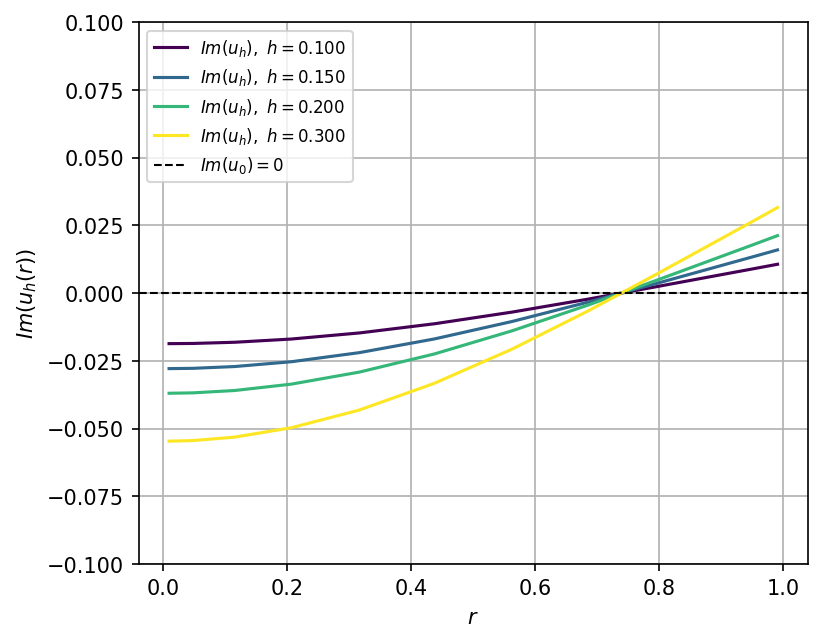}
    \end{subfigure}
    \caption{(Left) Convergence of $\Re(u_h)$ to $\Re(u_0) = u_0$ as $h\rightarrow0$. (Right) Convergence of $\Im(u_h)$ to $\Im(u_0)=0$ as $h\rightarrow0$.}
    \label{fig:mms18}
\end{figure}

We conclude with the incident-field problem introduced in \eqref{eq:nonlin-LS-incfield-def}. As noted at the end of Section \ref{sec:nonlin-solver}, the same numerical framework applies. Let $d\in\mathbb{S}^2$ be a propagation direction and consider the plane wave
$u^{\mathrm{inc}}(x)=\alpha e^{\iu\omega d\cdot x}$ with real amplitude $\alpha>0$ and physical frequency $\omega=\epsilon\hat\omega$.
To make the dependence on $\hat\omega$ explicit, set
\begin{equation}
     K(\hat\omega):=1-\hat\omega^2\innerproduct{\phi_j}{\K_D^{\epsilon\hat\omega}[\phi_j]}_D,\qquad
    M(\hat\omega):=\hat\omega^2\innerproduct{\phi_j}{\K_D^{\epsilon\hat\omega}[\abs{\phi_j}^2\phi_j]}_D,\qquad
    f:=\innerproduct{\phi_j}{u^{\mathrm{inc}}}_D.
\end{equation}
Using the definition \eqref{eq:def-rest-term-W}, equation \eqref{eq:unfolded-Pj-equation} becomes
\begin{equation}\label{eq:condensed-main-incfield-scalar-eq}
    (K(\hat\omega) - \eta M(\hat\omega) \abs{a}^2) a -f = \mathcal{R}_W.
\end{equation}
Recall that $\hat\omega_j=\lambda_j^{-1/2}$. The operator expansion \eqref{eq:prelim-K-volume-operator-expansion}, taken as $\epsilon\to0$ with $\hat\omega$ in a bounded neighbourhood of $\hat\omega_j$, gives
\begin{equation}\label{eq:helper-expansion-KDepsomega}
    \K_D^{\epsilon\hat\omega}
    =\K_D+\epsilon\hat\omega\K_{D,1}
    +O(\epsilon^2\abs{\hat\omega}^2)
\end{equation}
in operator norm. Since the kernel of $\K_D$ is real, $\phi_j$ may be chosen real, and hence
\begin{equation}
    \innerproduct{\phi_j}{\K_D^{\epsilon\hat\omega}[\phi_j]}_D
    =\lambda_j+\frac{\iu\epsilon\hat\omega}{4\pi}
    \left(\int_D\phi_j\,\dd x\right)^2
    +O(\epsilon^2\abs{\hat\omega}^2).
\end{equation}
For $\Delta:=\hat\omega-\hat\omega_j$, the identity $\hat\omega_j^2\lambda_j=1$ yields
\begin{equation}
    1-\hat\omega^2\lambda_j
    =-\frac{2}{\hat\omega_j}\Delta-\lambda_j\Delta^2.
\end{equation}
For the incident-field experiments below, we restrict attention to a mode satisfying $\int_D\phi_j\,\dd x\neq0$, as holds for the principal mode used numerically. Define the positive line width
\begin{equation}
    \Gamma:=\frac{\epsilon\hat\omega_j^4}{8\pi}
    \left(\int_D\phi_j\,\dd x\right)^2.
\end{equation}
It follows that
\begin{equation}\label{eq:leading-order-K-simplified-ampl-eq}
    K(\hat\omega)
    =-\frac{2}{\hat\omega_j}(\Delta+\iu\Gamma)
    +O(\abs{\Delta}^2+\epsilon\abs{\Delta}+\epsilon^2).
\end{equation}
Thus, keeping the scaled detuning $\mu:=\Delta/\Gamma$ fixed gives a natural line-width scale for studying \eqref{eq:condensed-main-incfield-scalar-eq} as $\tau\to\infty$. In the first experiment, shown in Figure \ref{fig:inc-field-exp1}, we fix the high contrast $\tau=200$ and sweep the incident frequency across this scaled detuning. Newton's method computes the solution of \eqref{eq:discrete-equation-incfield}. As expected, $\norm{u_h}_W^2$ closely follows $\abs{a_{\mathrm{th}}}^2$, where $a_{\mathrm{th}}$ is obtained from the leading-order relation \eqref{eq:simplified-amplitude-relation}.

\begin{figure}%
    \centering
    \includegraphics[width=0.9\textwidth]{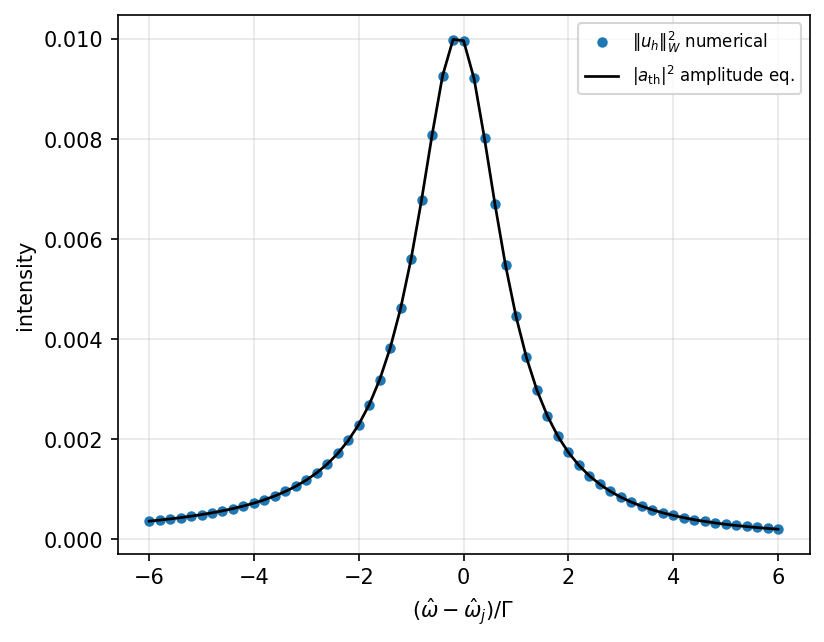}
    \caption{Intensity $\norm{u_h}_W^2$ of the numerical solution of \eqref{eq:discrete-equation-incfield} with fixed geometry and high contrast. The observed peak occurs near $\hat\omega=\hat\omega_j$. The amplitude $a_{\mathrm{th}}$ is computed from the reduced equation \eqref{eq:simplified-amplitude-relation}.}
    \label{fig:inc-field-exp1}
\end{figure}

The second experiment measures the ratio $\norm{W_h}_W/\abs{a_h}$ in the discrete decomposition $u_h=a_h\phi_j+W_h$. Here, the sampled eigenmode is normalized by $\norm{\phi_j}_W=1$, and we set
\begin{equation}
    P_{j,h}v:=\innerproduct{\phi_j}{v}_W\phi_j,
    \qquad Q_{j,h}:=I-P_{j,h},
    \qquad a_h:=\innerproduct{\phi_j}{u_h}_W,
    \qquad W_h:=Q_{j,h}u_h.
\end{equation}
We examine this ratio both as $\tau\to\infty$ and as $\abs{a_h}\to0$. Lemma \ref{lem:uniform-W-bound} suggests the scaling
\begin{equation}\label{eq:expected-numerical-scaling-Wa}
    \frac{\norm{W_h}_W}{\abs{a_h}} \lesssim \frac{\norm{Q_{j,h}u^\mathrm{inc}}_W}{\abs{a_h}} + \epsilon + \eta \abs{a_h}^2.
\end{equation}
For the plane wave above, we prescribe a value of $\mu=(\hat\omega-\hat\omega_j)/\Gamma$ and tune $\hat\omega=\hat\omega_j+\mu\Gamma$ as $\tau$ varies. With the incident amplitude $\alpha>0$ chosen through the reduced input--output relation, the computations use the estimate
\begin{equation}
    \norm{Q_{j,h}u^\mathrm{inc}}_W \leq C(\mu)\big(\epsilon\abs{a_h} + \eta \abs{a_h}^3\big),
\end{equation}
where the constant depends on the prescribed scaled detuning. This motivates the following two empirical scaling laws tested numerically:
\begin{equation}\label{eq:two-case-expected-error}
\begin{aligned}
    \frac{\norm{W_h}_W}{\abs{a_h}} &\leq C(\mu)\big(\epsilon + \eta \abs{a_h}^2\big), \qquad \text{for fixed }\mu\neq0, \\
     \frac{\norm{W_h}_W}{\abs{a_h}} &\leq C_1\epsilon^2 + C_2 \eta \abs{a_h}^2, \qquad \text{for }\mu=0.
    \end{aligned}
\end{equation}
While the first scaling law is consistent with the rigorous estimate in Lemma \ref{lem:uniform-W-bound}, the second reflects an additional numerically observed cancellation. 
Figure \ref{fig:inc-field-exp2} illustrates both cases in \eqref{eq:two-case-expected-error}. The left panel compares the exact resonance, $(\hat\omega-\hat\omega_j)/\Gamma=0$, with the fixed nonzero detuning $(\hat\omega-\hat\omega_j)/\Gamma=3$ as $\tau\to\infty$. The observed slopes are in agreement with $\epsilon^2=\tau^{-1}$ and $\epsilon=\tau^{-1/2}$, respectively. The right panel fixes $\hat\omega=\hat\omega_j$ and examines the ratio as $\abs{a_h}\to0$ at $\tau=200$ and $5000$; the dashed curves are least-squares fits of $C_1\epsilon^2+C_2\eta\abs{a_h}^2$.

\begin{figure}%
    \centering
    \includegraphics[width=0.9\textwidth]{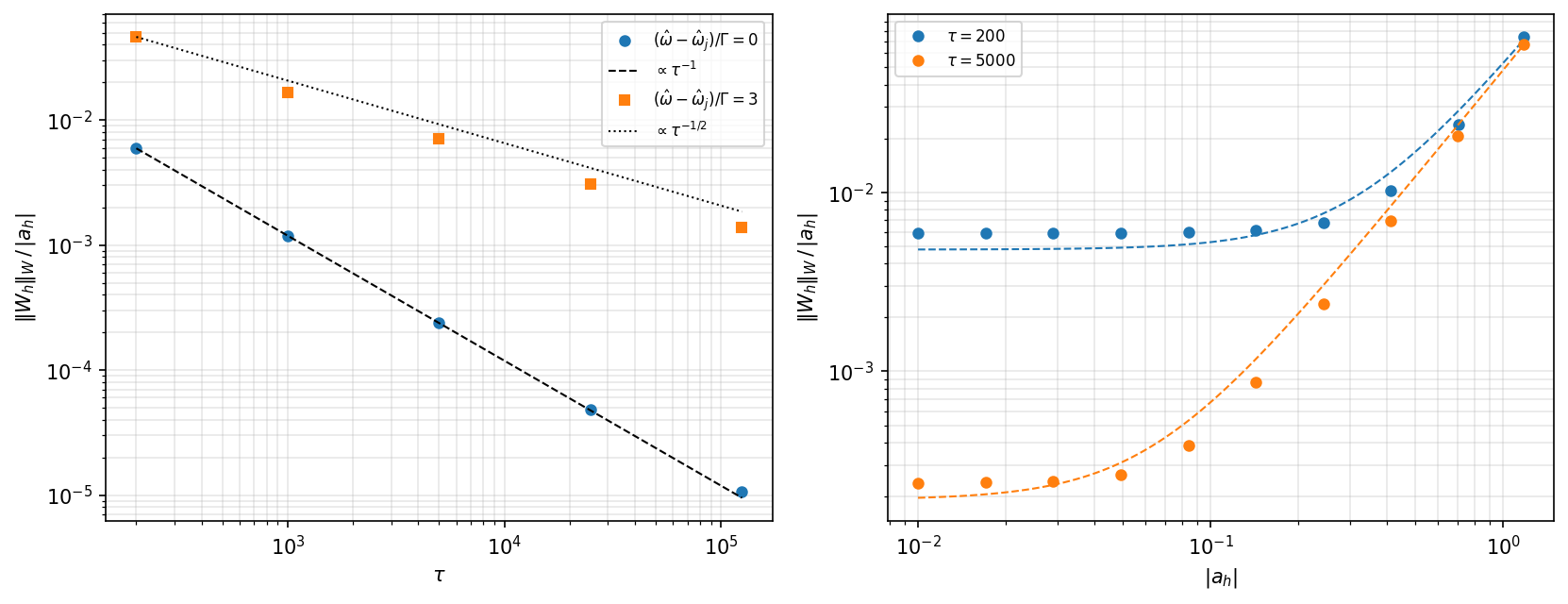}
    \caption{Remainder ratio $\norm{W_h}_W/\abs{a_h}$ for $u_h=a_h\phi_j+W_h$. (Left) Contrast dependence at exact resonance and at a fixed nonzero scaled detuning. (Right) Amplitude dependence at exact resonance for two high contrasts; dashed curves show least-squares fits of $C_1\epsilon^2+C_2\eta\abs{a_h}^2$.}
    \label{fig:inc-field-exp2}
\end{figure}
The third experiment examines the multivalued response predicted by \eqref{eq:simplified-amplitude-relation} at negative scaled detuning. Fix $\mu:=\Delta/\Gamma$. Since $\Delta=O(\epsilon)$ in this regime, \eqref{eq:leading-order-K-simplified-ampl-eq} becomes
\begin{equation}
    K(\hat\omega)=-\frac{2}{\hat\omega_j}(\Delta+\iu\Gamma)+O(\epsilon^2).
\end{equation}
Similarly, \eqref{eq:helper-expansion-KDepsomega} gives
\begin{equation}\label{eq:expansion-M-amplitude-eq}
    M(\hat\omega)
    =\hat\omega_j^2\innerproduct{\phi_j}{\K_D[\phi_j^3]}_D
    +\frac{\iu\epsilon\hat\omega^3}{4\pi}
    \left(\int_D\phi_j\,\dd x\right)
    \left(\int_D\phi_j^3\,\dd x\right)
    +O(\epsilon^2+\abs{\Delta}).
\end{equation}
Because $\K_D$ is real, $\phi_j$ can be chosen to be real-valued. By self-adjointness and $\K_D\phi_j=\lambda_j\phi_j$, the leading term is the positive constant
\begin{equation}
    M_0:=\hat\omega_j^2\innerproduct{\phi_j}{\K_D[\phi_j^3]}_D
    =\hat\omega_j^2\innerproduct{\K_D[\phi_j]}{\phi_j^3}_D
    =\hat\omega_j^2\lambda_j\int_D\phi_j^4\,\dd x
    =\int_D\phi_j^4\,\dd x>0.
\end{equation}
Thus, $M(\hat\omega)=M_0+O(\epsilon)$ for fixed $\mu$. Substitution into \eqref{eq:conditions-distinct-roots} shows that, at leading order, its two conditions reduce to
\begin{equation}
    \mu<0,
    \qquad
    \mu^2>3.
\end{equation}
Hence, at leading order in $\epsilon$, the input--output relation can be multivalued when $\mu<-\sqrt{3}$. In Figure \ref{fig:inc-field-exp4}, we fix $\tau=200$ and $\mu=-4$. The solid curve shows the response predicted by \eqref{eq:simplified-amplitude-relation}, while the circles show solutions of the full discrete equation \eqref{eq:discrete-equation-incfield}.

For these computations, we take $u^{\mathrm{inc}}(x)=\alpha e^{\iu\omega d\cdot x}$ with $\alpha>0$ and $d=\vb e_z=(0,0,1)$. Let
\begin{equation}
    E_\omega:=\big(e^{\iu\omega d\cdot X_k}\big)_{k=1}^n,
    \qquad
    f_h:=\alpha\innerproduct{\phi_j}{E_\omega}_W.
\end{equation}
For the discrete reduced residual, define
\begin{equation}
    \kappa_h(\hat\omega)
    :=1-\hat\omega^2\innerproduct{\phi_j}{K_h^{\epsilon\hat\omega}[\phi_j]}_W,
    \qquad
    m_h(\hat\omega)
    :=\hat\omega^2\innerproduct{\phi_j}{K_h^{\epsilon\hat\omega}[\abs{\phi_j}^2\phi_j]}_W.
\end{equation}
Thus, the plotted coordinates are
\begin{equation}
    (\abs{f_h}^2,\abs{a_h}^2)
    =\left(\alpha^2\abs{\innerproduct{\phi_j}{E_\omega}_W}^2,
    \abs{\innerproduct{\phi_j}{u_h}_W}^2\right).
\end{equation}
We apply PALC to the incident-field system with $x=(\Re u,\Im u)$ and $p=\alpha$; unlike \eqref{eq:palc-map}, this system has neither a normalization nor a phase constraint. The computed points follow the predicted S-shaped curve, and the red dots mark three distinct solutions at the same value of $\abs{f_h}^2$. This comparison is numerical, and the larger-amplitude part of the curve may lie outside the proved local regime.

\begin{figure}%
    \centering
    \includegraphics[width=0.9\textwidth]{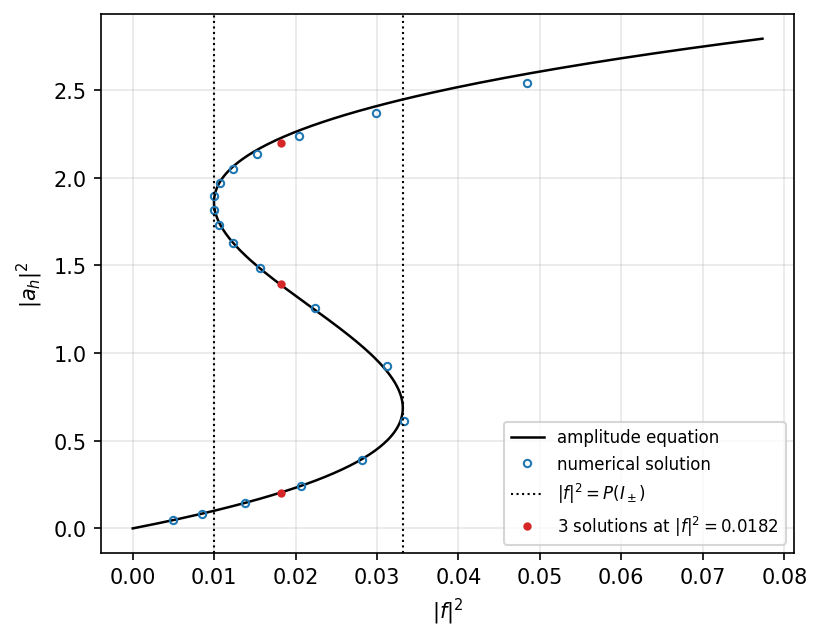}
    \caption{Multivalued input--output response at $(\hat\omega-\hat\omega_j)/\Gamma=-4$. The solid curve is predicted by \eqref{eq:simplified-amplitude-relation}, the circles are solutions of the full discrete problem, and the red dots mark three solutions at the same input intensity.}
    \label{fig:inc-field-exp4}
\end{figure}

Finally, a fourth experiment tests the interior local decomposition, the reduced amplitude equation, and the exterior estimate from Proposition \ref{prop:scattered-final-prop}. For contrasts ranging from $200$ to order $10^5$, we compute a solution $u_h$ of \eqref{eq:discrete-equation-incfield}. We fix the scaled detuning $(\hat\omega-\hat\omega_j)/\Gamma=3$ and choose the incident amplitude through the reduced relation so that $\abs{a_h}\approx0.01$.

At fixed amplitude, the nonlinear contribution eventually dominates when $\epsilon<\eta\abs{a_h}^2$. Thus, maintaining the slopes in Figure \ref{fig:inc-field-exp3} in a strict $\tau\to\infty$ limit would require $\abs{a_h}^2=O(\epsilon)$. We keep $\abs{a_h}\approx0.01$ for consistency with the preceding experiments; the nonlinear contribution remains negligible over the tested range.

In Figure \ref{fig:inc-field-exp3}, the blue curve shows
\begin{equation}
    \frac{\norm{u_h^{\mathrm{sc}}-(a_{\mathrm{th}}\phi_j-u^{\mathrm{inc}})}_W}
    {\norm{u_h^{\mathrm{sc}}}_W},
\end{equation}
where $a_{\mathrm{th}}$ solves \eqref{eq:simplified-amplitude-relation}. The green curve shows the normalized residual
$\abs{(\kappa_h-\eta m_h\abs{a_h}^2)a_h-f_h}/\abs{a_h}$, the discrete counterpart of $\abs{\mathcal{R}_W}/\abs{a}$. These are interior tests of the Lyapunov--Schmidt reduction. The orange curve tests the exterior estimate from Proposition \ref{prop:scattered-final-prop} in the weaker $L^2$ norm. Denote by $\Phi_{j,h}^{\epsilon\hat\omega}$ the quadrature approximation of $\Phi_j^{\epsilon\hat\omega}$ from \eqref{eq:def-exterior-terms}. The orange curve shows the relative discrete $L^2(\Omega)$ error between the numerical exterior scattered field $u_h^{\mathrm{sc}}$ and $a_h\Phi_{j,h}^{\epsilon\hat\omega}$. Over the tested high-contrast range, this error decays at the expected $O(\epsilon)$ rate because the term involving $\eta$ remains negligible. We take
\begin{equation}
    \Omega:=\left\{x\in\R^3:\frac{3}{2}R<\abs{x}<\frac{5}{2}R\right\},
    \qquad R=1,
\end{equation}
so that $\Omega\Subset\R^3\setminus\overline D$.

\begin{figure}%
    \centering
    \includegraphics[width=0.9\textwidth]{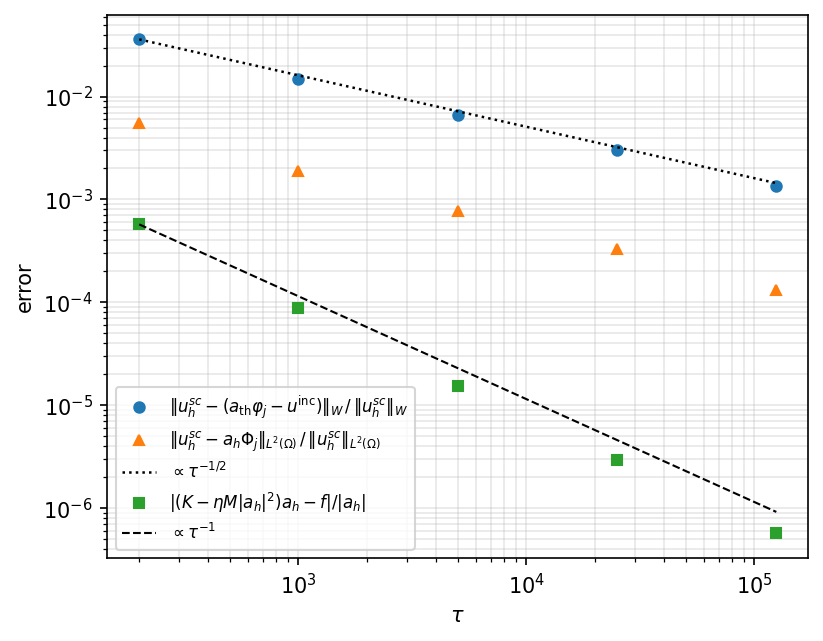}
    \caption{Incident-field tests as $\tau\to\infty$. (Blue) Relative interior scattered-field error for the approximation $a_{\mathrm{th}}\phi_j-u^{\mathrm{inc}}$. (Orange) Relative discrete exterior $L^2(\Omega)$ error for the leading approximation $a_h\Phi_{j,h}^{\epsilon\hat\omega}$. (Green) Normalized residual of the reduced amplitude equation evaluated at $a_h=\innerproduct{\phi_j}{u_h}_W$.}
    \label{fig:inc-field-exp3}
\end{figure}

\subsection{Dimer of Resonators}\label{sec:dimer}

In this section, we validate and extend the bifurcation machinery developed in the previous sections and show that the numerical framework of a dimer of resonators robustly captures symmetry-broken nonlinear resonant modes. 

For doing so, we consider the mirror-symmetric dimer $D=D_1\cup D_2$ introduced in \cite[Section 4]{ammari_dielectric_2025}, with symmetry hyperplane $\{x_1=0\}$. Each spherical resonator has radius $R=1$, and $g=2L-2R$ denotes the horizontal gap between them. We study a range of gap ratios $g/R$, including the near-touching regime.
Throughout, we work in the three-dimensional high-contrast regime under Assumption A of \cite{ammari_dielectric_2025}. Let $(\lambda_+,\phi_+)$ and $(\lambda_-,\phi_-)$ be the two leading eigenpairs of the Newtonian potential $\K_D$, where $\lambda_+>\lambda_->0$. If $\mathcal{R}$ denotes reflection across $\{x_1=0\}$, then $\phi_+$ is symmetric and $\phi_-$ is antisymmetric:
\begin{equation}
    \mathcal{R}[\phi_+] = \phi_+,
    \qquad
    \mathcal{R}[\phi_-] = -\phi_-.
\end{equation}
The eigenfunctions can be chosen to be real because the kernel of $\K_D$ is real. We
use the overlap coefficients
\begin{equation}
    A_{++}:=\norm{\phi_+^2}_{L^2(D)}^2
    =\int_D\phi_+^4\,\dd x,
    \qquad
    A_{+-}:=\norm{\phi_+\phi_-}_{L^2(D)}^2
    =\int_D\phi_+^2\phi_-^2\,\dd x.
\end{equation}
Under Assumptions A and B of \cite{ammari_dielectric_2025}, the local bifurcation results in \cite[Cor. 4.3 and Thm. 4.12]{ammari_dielectric_2025} yield primary symmetric and antisymmetric branches bifurcating from the linear resonances $\omega_{*,\pm}$, as well as a secondary symmetry-breaking branch near the critical normalization
\begin{equation}\label{eq:Ncrit}
    \mathcal{N}_{\mathrm{crit}} \sim \dfrac{\lambda_+ - \lambda_-}{3\lambda_- A_{+-} - \lambda_+ A_{++}}.
\end{equation}
The corresponding leading critical symmetric-mode amplitude satisfies
$p_{+,*}^2\sim\mathcal{N}_{\mathrm{crit}}$.
Our aim is to confirm these bifurcation phenomena numerically and study their dependence on the separation distance. Taking $\eta=1$, we track the asymmetric branches by introducing a small imperfection in the \emph{linear} coefficient of the two resonators:
\begin{equation}\label{eq:linear-imperfection}
    u - \tau\omega^2 \K_D^\omega[c_\delta u + \abs{u}^2u] = 0,
    \qquad
    c_\delta(x):=
    \begin{cases}
        1+\delta, & x\in D_1,\\
        1-\delta, & x\in D_2.
    \end{cases}
\end{equation}
Here, $\delta>0$ is small; exchanging $D_1$ and $D_2$ gives the opposite imperfection. For a fixed distance, we compute the symmetric solution $(u_{\mathrm{sym}}(\mathcal{N}),\omega(\mathcal{N}))$ and use it as the initial guess for the two imperfect continuations. We then solve the full dimer problem and compute the asymmetry function
\begin{equation}\label{eq:asymm-function}
    \begin{aligned}
        A(\mathcal{N}_k) &= \frac{P_1 - P_2}{P_2 + P_1}, \\
        P_j &= \sum_{i \in \mathcal{I}_j}W_i\abs{u_i}^2,
    \end{aligned}
\end{equation}
where $\mathcal{I}_j$ is the set of indices belonging to the resonator $D_j$ and $\mathcal{N}_k$ is the normalization constant at which symmetry breaking is tested. Near a secondary bifurcation, the imperfect continuations develop pronounced asymmetry and, as $\delta\to0$, approach the two asymmetric branches. The two exchanged imperfections produce opposite signs of $A$.
The asymmetry function distinguishes the two asymmetric branches, which have the same real part $\Re(\omega)$ and therefore overlap in a standard bifurcation diagram.
Using Newton's method for the nonlinear Lippmann--Schwinger equation, initialized from normalized linear resonant pairs, we numerically confirm symmetry breaking and track all four branches at several separation distances; see Figures \ref{fig:dimer-bifdiag-g005}, \ref{fig:dimer-bifdiag-g1-g01}, and \ref{fig:dimer-bifdiag-g10-g19}. Additional asymmetry-function plots are presented in Appendix \ref{app:appendixB}.

\begin{figure}%
    \centering
    \begin{subfigure}[b]{\linewidth}
        \centering
        \includegraphics[width=0.9\textwidth]{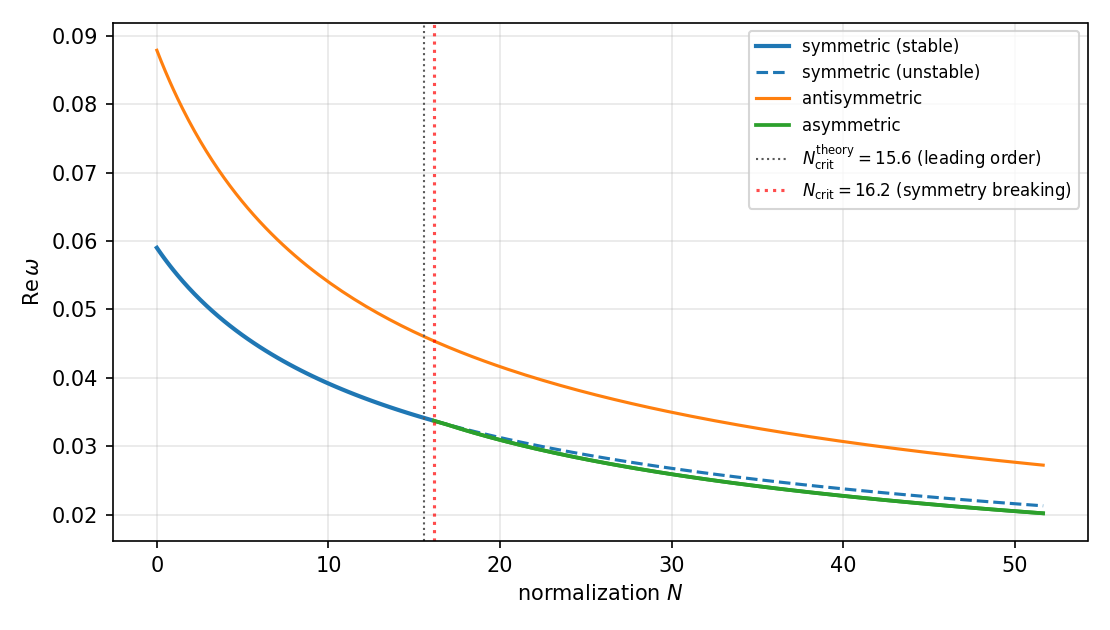}
    \end{subfigure}
    \begin{subfigure}[b]{\linewidth}
        \centering
        \includegraphics[width=0.9\textwidth]{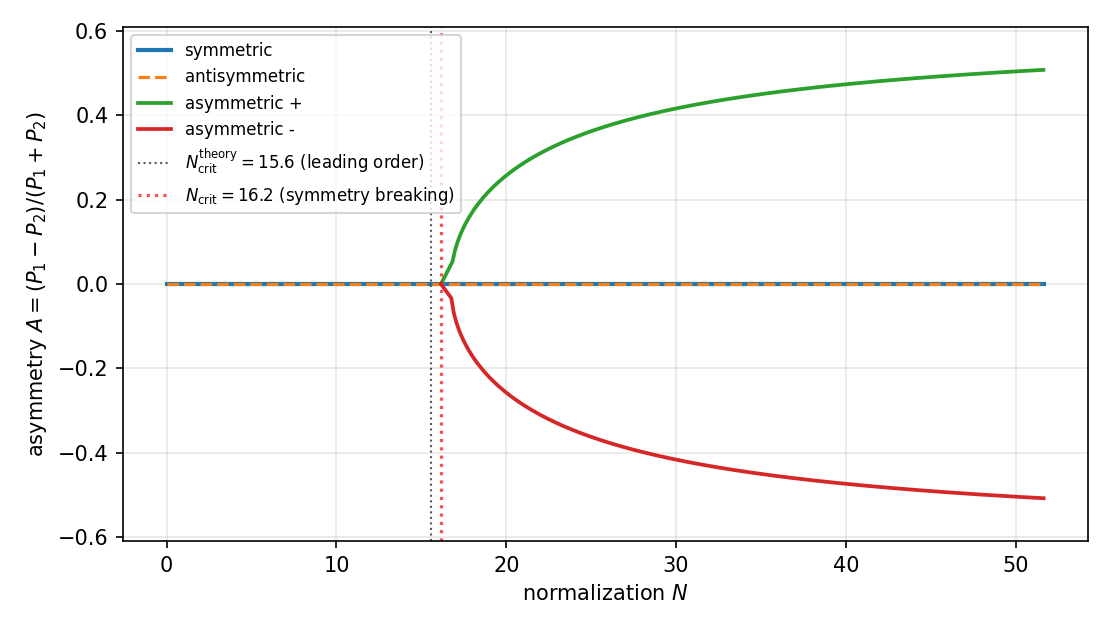}
    \end{subfigure}
    \caption{Dimer of resonators with small gap ratio $g/R = 0.05$, where $R=1$ is the radius and $g=2L-2R$ is the gap. (Top) Bifurcation diagram tracking all four branches as the normalization constant increases. The two asymmetric branches have equal $\Re(\omega)$, so they overlap. Both the leading-order critical normalization and the observed symmetry-breaking threshold are indicated. (Bottom) Asymmetry function showing that both asymmetric branches are identified and tracked.}
    \label{fig:dimer-bifdiag-g005}
\end{figure}

\begin{figure}%
    \centering
    \begin{subfigure}[b]{0.45\textwidth}
        \centering
        \includegraphics[width=\linewidth]{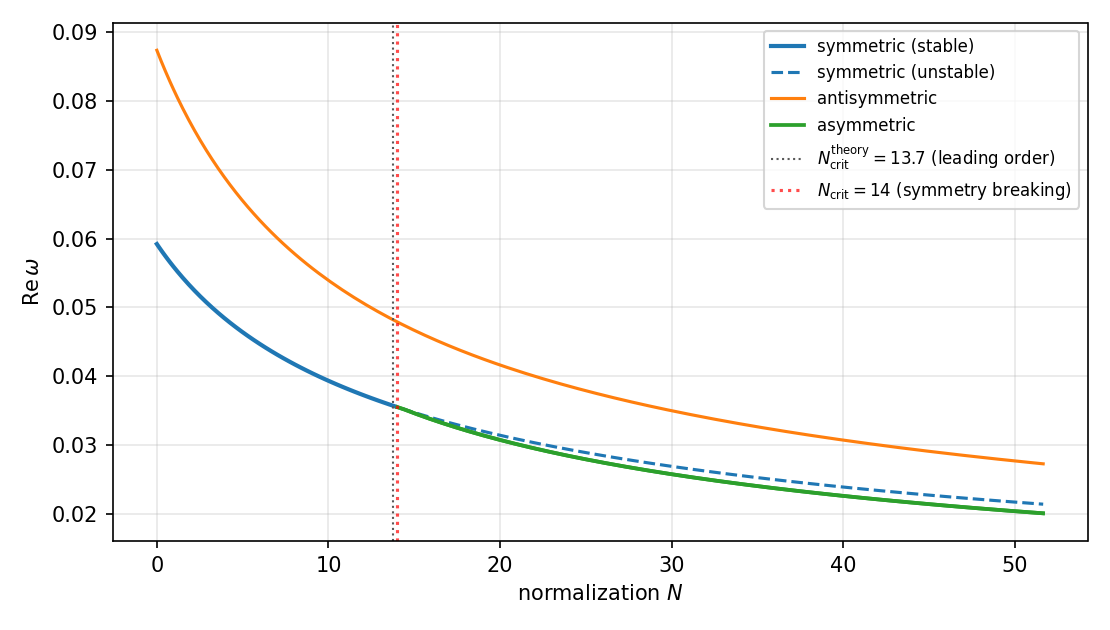}
        \caption{Dimer with small gap ratio $g/R = 0.1$.}
        \label{fig:dimer-bifdiag-g01}
    \end{subfigure}
    \begin{subfigure}[b]{0.45\textwidth}
        \centering
        \includegraphics[width=\linewidth]{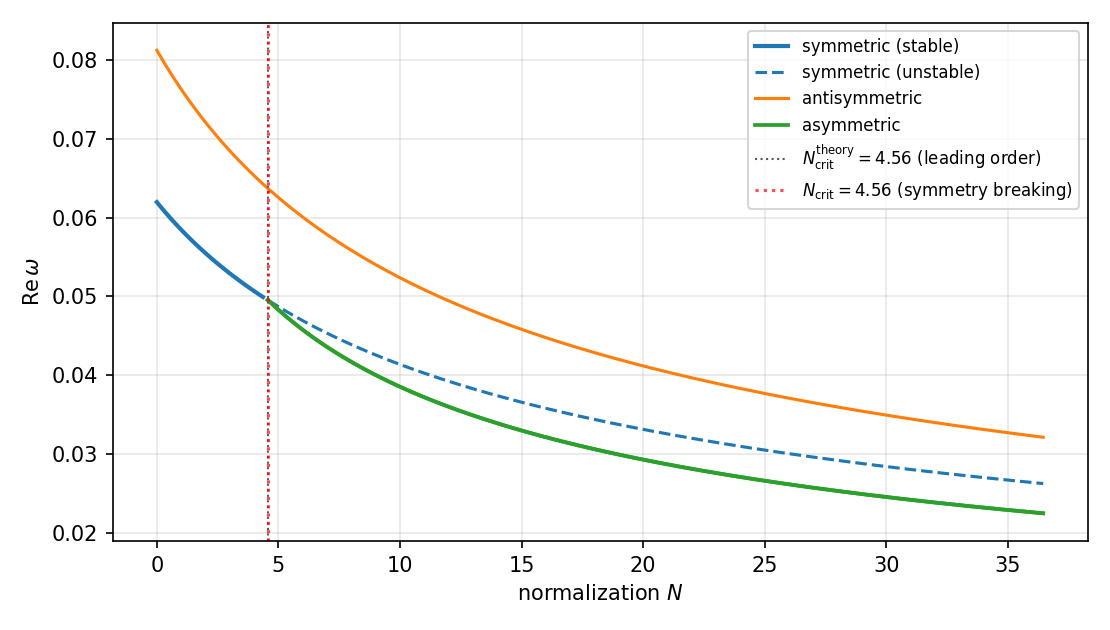}
        \caption{Dimer of resonators with gap ratio $g/R=1$.}
        \label{fig:dimer-bifdiag-g1}
    \end{subfigure}
    \caption{ diagram tracking all four branches as the normalization constant increases. (Left) Small-gap regime. (Right) The gap equals the resonator radius.}
    \label{fig:dimer-bifdiag-g1-g01}
\end{figure}

\begin{figure}%
    \centering
    \begin{subfigure}[b]{0.45\textwidth}
        \centering
        \includegraphics[width=\linewidth]{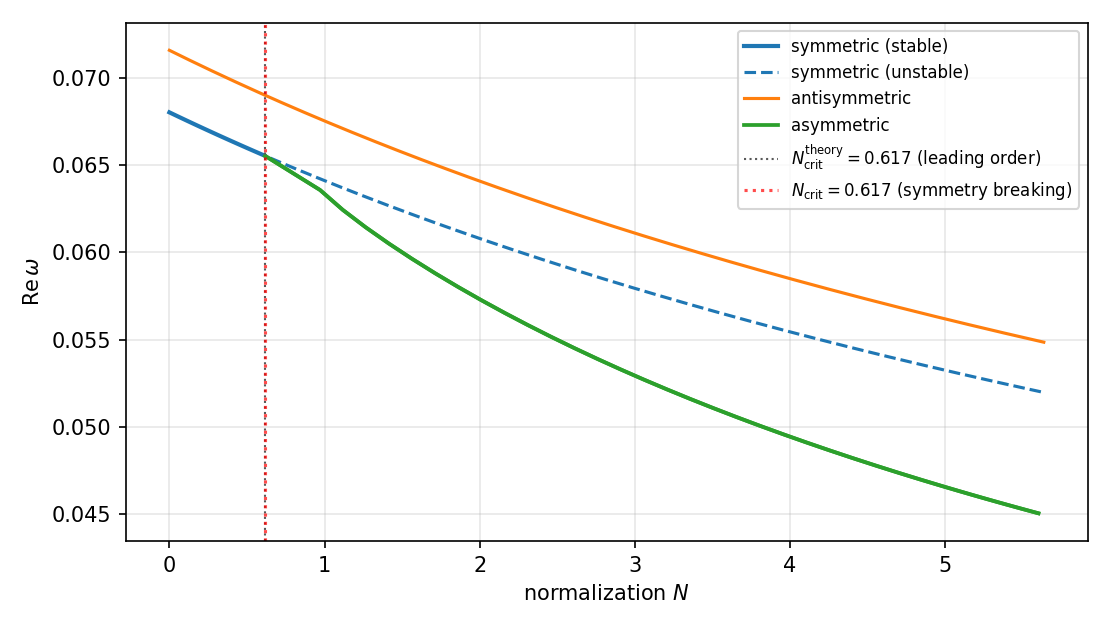}
        \caption{Dimer of resonators with gap ratio $g/R=10$.}
        \label{fig:dimer-bifdiag-g10}
    \end{subfigure}
    \begin{subfigure}[b]{0.45\textwidth}
        \centering
        \includegraphics[width=\linewidth]{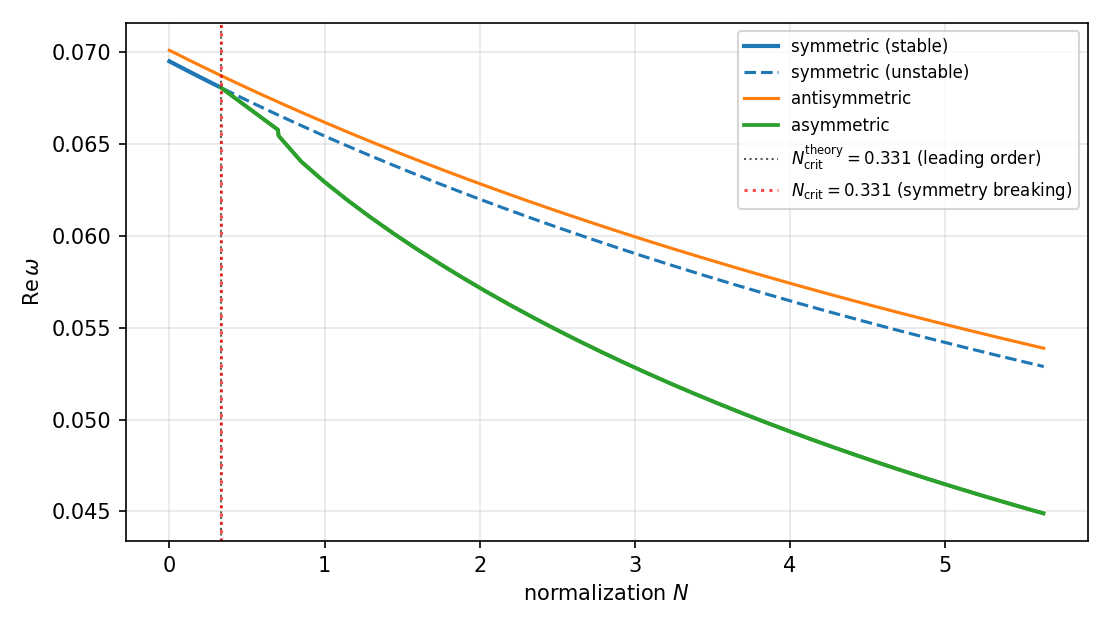}
        \caption{Dimer of resonators with gap ratio $g/R=19$.}
        \label{fig:dimer-bifdiag-g19}
    \end{subfigure}
    \caption{Bifurcation diagrams for a resonator dimer at large separation distances, corresponding to the dilute regime.}
    \label{fig:dimer-bifdiag-g10-g19}
\end{figure}

Figure \ref{fig:dimer-bifdiag-g005} shows that the solver tracks all four branches even at a very small separation distance. Because the asymmetric branches have equal real parts, the accompanying asymmetry plot is needed to distinguish them. Figures \ref{fig:dimer-bifdiag-g005}, \ref{fig:dimer-bifdiag-g1-g01}, and \ref{fig:dimer-bifdiag-g10-g19}, computed at contrast $\tau=500$ with quadrature $(n_r,n_\theta,n_\phi)=(8,8,12)$, summarize the behaviour across the range of gaps.
The symmetric and antisymmetric linear modes satisfy
\begin{equation}\label{eq:approx-linear-modes}
    \omega_{*,\pm}\approx\frac{1}{\sqrt{\tau\lambda_\pm}}.
\end{equation}
Since $\lambda_+>\lambda_-$, this gives $\Re(\omega_{*,+})<\Re(\omega_{*,-})$. Moreover, at the limiting bifurcation point $\epsilon=0$, the expansion used in the proof of \cite[Theorem 4.12]{ammari_dielectric_2025} gives
\begin{equation}
    \hat\omega_*^{-2}
    = \lambda_+\big(1+p_{+,*}^2A_{++}\big)
    +\mathcal{O}(p_{+,*}^4),
\end{equation}
which explains the observed decrease of the resonant frequency as the amplitude increases. The mode splitting also decreases as the separation distance grows, in agreement with \cite[Proposition 4.4]{ammari_dielectric_2025}. In the dilute regime shown in Figure \ref{fig:dimer-bifdiag-g10-g19}, even and odd modes approach, respectively, the symmetric and antisymmetric combinations of translated copies of the principal single-particle state. More precisely,
\begin{equation}
    \lambda_\pm=\lambda_0\pm k_I(L)+\mathcal{O}(L^{-2}),
    \qquad k_I(L)=\mathcal{O}(L^{-1}),
\end{equation}
where $\lambda_0$ is the principal eigenvalue of a single particle. Let
$\phi_0$ be its normalized eigenfunction and define the translation operator
$(\mathcal{T}_L\phi)(x_1,x_2,x_3):=\phi(x_1+L,x_2,x_3)$.
Then the interaction term is
\begin{equation}
    k_I(L):=\innerproduct{\mathcal{R}\mathcal{T}_L\phi_0}
    {\K_D[\mathcal{T}_L\phi_0]}_D>0.
\end{equation}
The associated normalized eigenfunctions satisfy
\begin{equation}
    \left\|\phi_\pm-\frac{\mathcal{T}_L\phi_0
    \pm\mathcal{R}\mathcal{T}_L\phi_0}{\sqrt{2}}\right\|_{L^2(D)}
    \longrightarrow0
    \qquad\text{as }L\to\infty.
\end{equation}

Finally, the computations reveal how the symmetry-breaking threshold changes as the gap narrows. Equation \eqref{eq:Ncrit} gives the leading-order prediction shown by the dotted red lines in the bifurcation diagrams, and the observed numerical onset of symmetry breaking is close to this prediction. For the smaller gaps in Figures \ref{fig:dimer-bifdiag-g005} and \ref{fig:dimer-bifdiag-g01}, the observed threshold increases as the gap decreases, while remaining close to the leading-order value. Together with the gap dependence predicted by \eqref{eq:Ncrit} and shown in Figure \ref{fig:dimer-crit-threshold-vs-gap}, these results suggest that stronger hybridization makes it harder to achieve symmetry breaking.

\begin{figure}%
    \centering
    \includegraphics[width=0.9\linewidth]{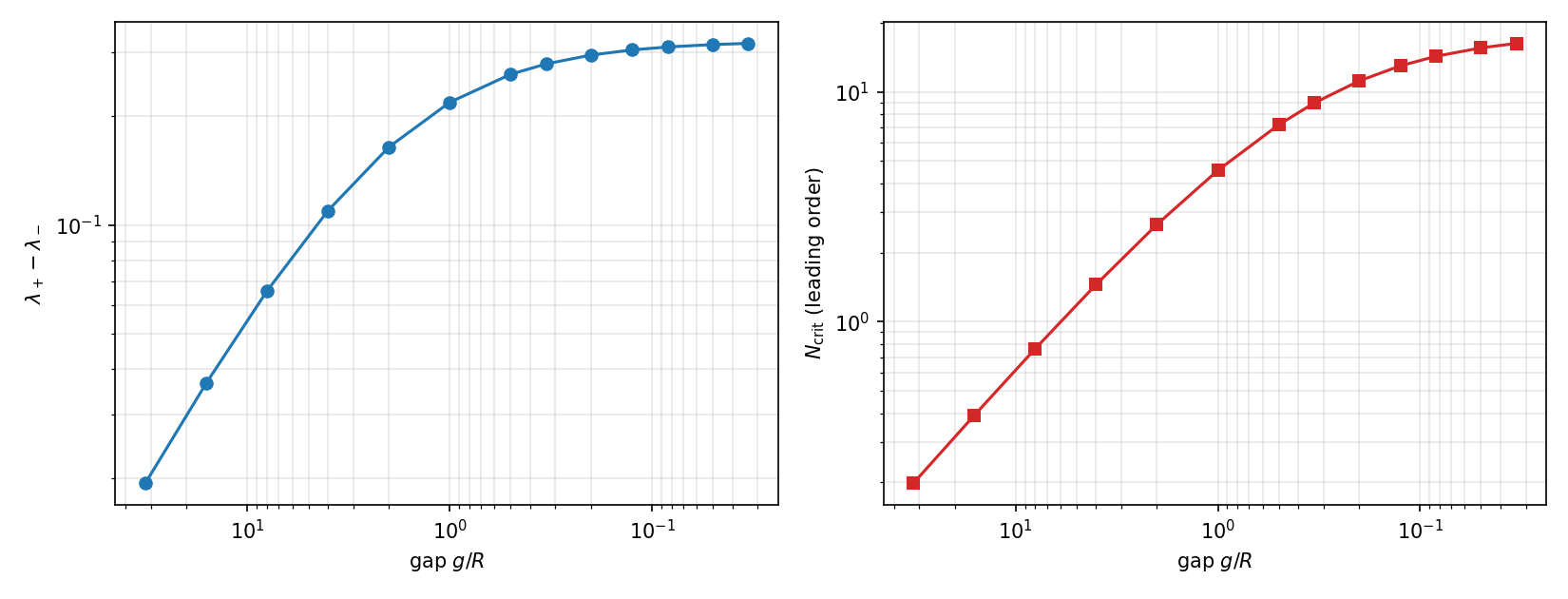}
    \caption{Behaviour of the leading-order critical normalization from \eqref{eq:Ncrit} as the gap becomes small. (Left) Eigenvalue splitting $\lambda_+-\lambda_-$, which is the numerator in \eqref{eq:Ncrit}. (Right) The corresponding value of $\mathcal{N}_{\mathrm{crit}}$.}
    \label{fig:dimer-crit-threshold-vs-gap}
\end{figure}

\section{Concluding Remarks}\label{sec:conclusion}

In this work, we have developed a nonlinear modal reduction for high-contrast subwavelength dielectric scattering and combined it with a numerical continuation framework for nonlinear dielectric resonances. Our analysis shows that, near a simple eigenmode of the Newtonian potential, the nonlinear Lippmann--Schwinger equation can be reduced locally to a single complex modal amplitude. The complementary component is uniquely determined and quantitatively controlled, while the resonant amplitude satisfies an exact scalar equation whose leading part has a cubic form. This provides a nonlinear counterpart of the single-mode decompositions used in linear subwavelength scattering.

The reduction also yields a direct approximation of the scattered field. Away from the resonator, the scattered field is given to leading order by the outgoing field generated by the resonant Newtonian mode, with a remainder controlled by the nonresonant part of the incident field, the high-contrast parameter, and the cubic nonlinear correction. This identifies a precise local regime in which a nonlinear resonator can be described by one effective complex degree of freedom.

The leading amplitude equation naturally exhibits nonlinear multistability. For a simple mode with nonzero monopole moment, expansion near the linear resonance identifies the radiative linewidth and shows that the turning-point condition is governed by the detuning measured on this linewidth scale. The numerical incident-wave experiments reproduce the predicted line shape, confirm the modal remainder estimates over the tested high-contrast range, and display an S-shaped input--output curve with three distinct solutions at the same incident intensity.

The numerical continuation framework also captures symmetry-breaking bifurcations in a resonator dimer. By combining a phase condition, normalization, imperfection continuation, and pseudo-arclength continuation, we track the symmetric, antisymmetric, and two asymmetric branches over gap ratios ranging from the near-touching to the dilute regime. The observed bifurcation threshold remains close to the leading-order theory and increases as the resonators approach one another. This indicates that a stronger hybridization increases the nonlinear intensity required to destabilize the symmetric branch.

Our results in this paper provide a foundation for studying the time-dependent stability of principal nonlinear resonant modes. Determining the dynamical stability of nonlinear resonant modes requires coupling the modal reduction to an appropriate time-dependent model and analyzing the spectrum of the linearization around each branch. In particular, it would be very interesting to relate the turning points and symmetry-breaking points in the stationary diagrams to dynamical instabilities and switching phenomena. Using our proposed analytical and numerical framework, we propose to analyze larger dielectric-resonator systems, such as clusters and periodic resonator systems. In that setting, the scalar amplitude equation should be replaced by a nonlinear system, allowing mode competition, internal resonances, and symmetry-induced degeneracies to be treated within the same framework. These problems will be the subject of future work.

\section*{Code Availability}
\noindent
The code used to obtain the numerical results of this paper is openly available at \url{https://github.com/goriwastaken/nonlinear-dielectric-resonators}. 

\section*{Acknowledgments} 
\noindent 
This work was partially supported by the City University of Hong Kong start-up fund 7200843 and the Swiss National Science Foundation grant number 200021-236472. 

\appendix

\section{Additional Numerical Results}\label{app:appendixB}
We present the asymmetry-function plots corresponding to the bifurcation diagrams in Figures \ref{fig:dimer-bifdiag-g01}, \ref{fig:dimer-bifdiag-g1}, \ref{fig:dimer-bifdiag-g10}, and \ref{fig:dimer-bifdiag-g19}. They show that the solver identifies both asymmetric branches bifurcating from the symmetric branch; the two branches otherwise appear as one line in the bifurcation diagrams because their real parts are equal.

\begin{figure}[!ht]
    \centering
    \includegraphics[width=0.9\textwidth]{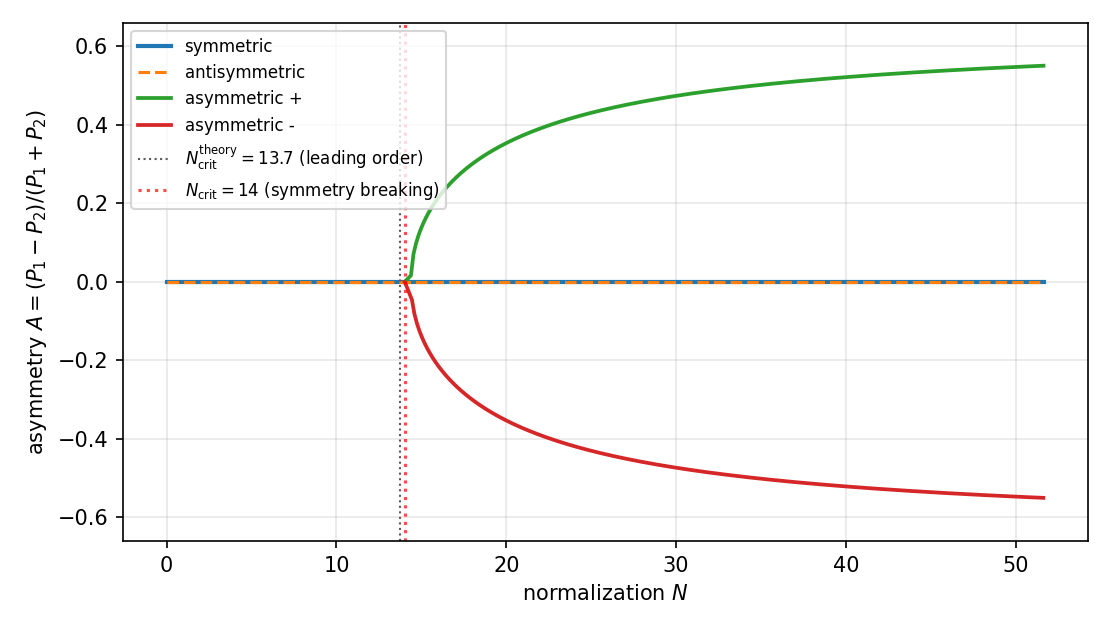}
    \caption{Asymmetry function plot for a dimer with gap ratio $g/R = 0.1$, visualizing the two different asymmetric branches from Figure \ref{fig:dimer-bifdiag-g01}.}
    \label{fig:dimer-bifdiag-asym-g01}
\end{figure}

\begin{figure}%
    \centering
    \includegraphics[width=0.9\textwidth]{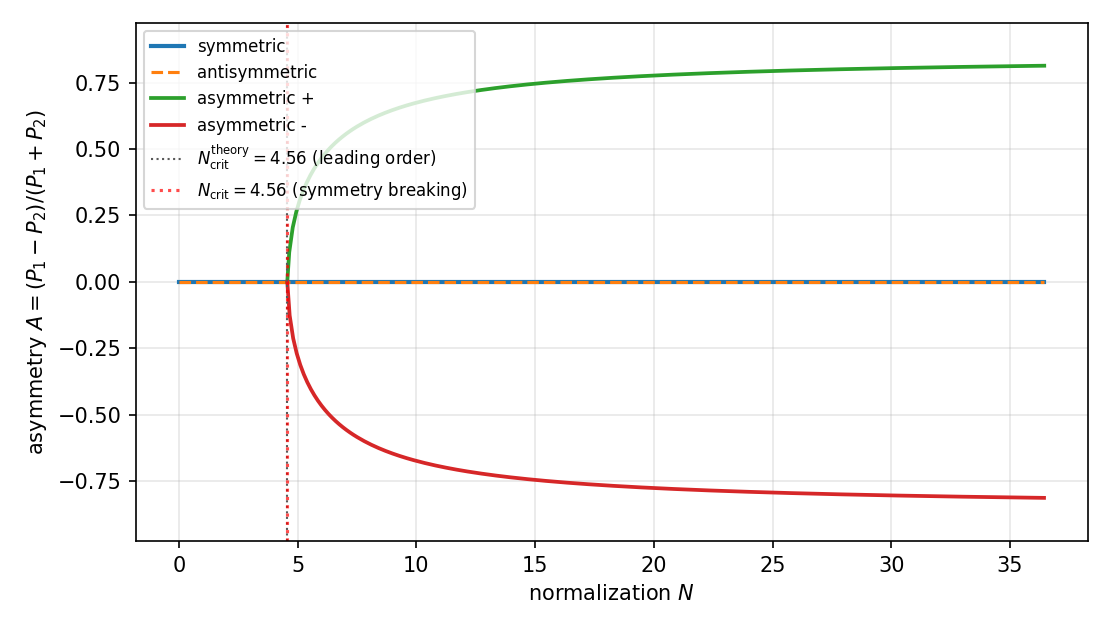}
    \caption{Asymmetry function plot for a dimer with gap ratio $g/R = 1$, visualizing the two different asymmetric branches from Figure \ref{fig:dimer-bifdiag-g1}.}
    \label{fig:dimer-bifdiag-asym-g1}
\end{figure}

\begin{figure}%
    \centering
    \includegraphics[width=0.9\textwidth]{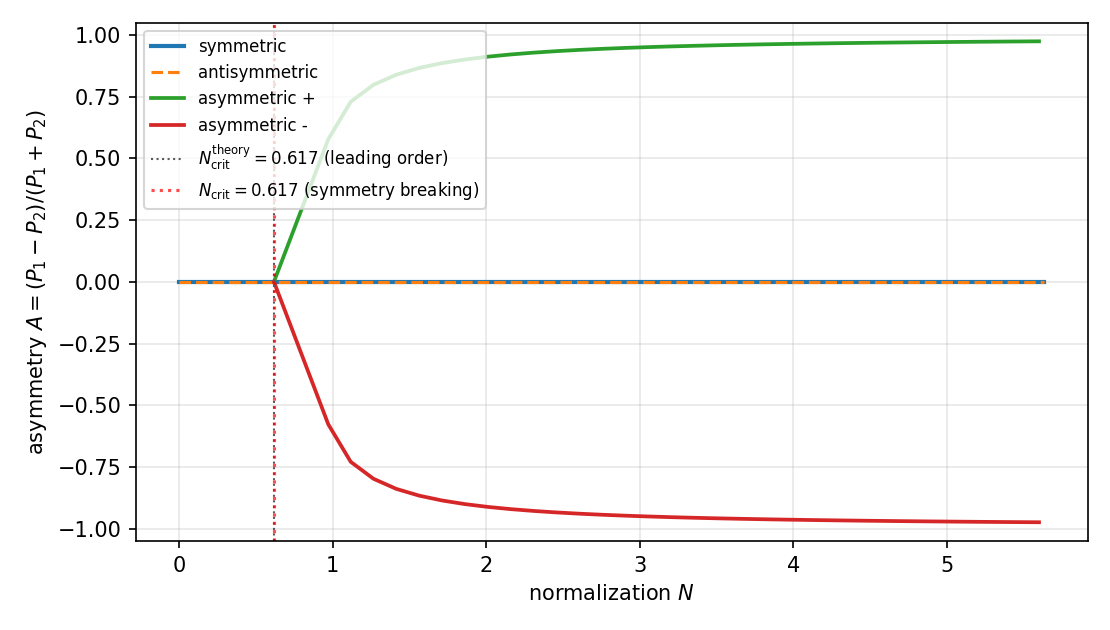}
    \caption{Asymmetry function plot for a dimer with gap ratio $g/R = 10$, distinguishing the two asymmetric branches shown in Figure \ref{fig:dimer-bifdiag-g10}.}
    \label{fig:dimer-bifdiag-asym-g10}
\end{figure}
\begin{figure}%
    \centering
    \includegraphics[width=0.9\textwidth]{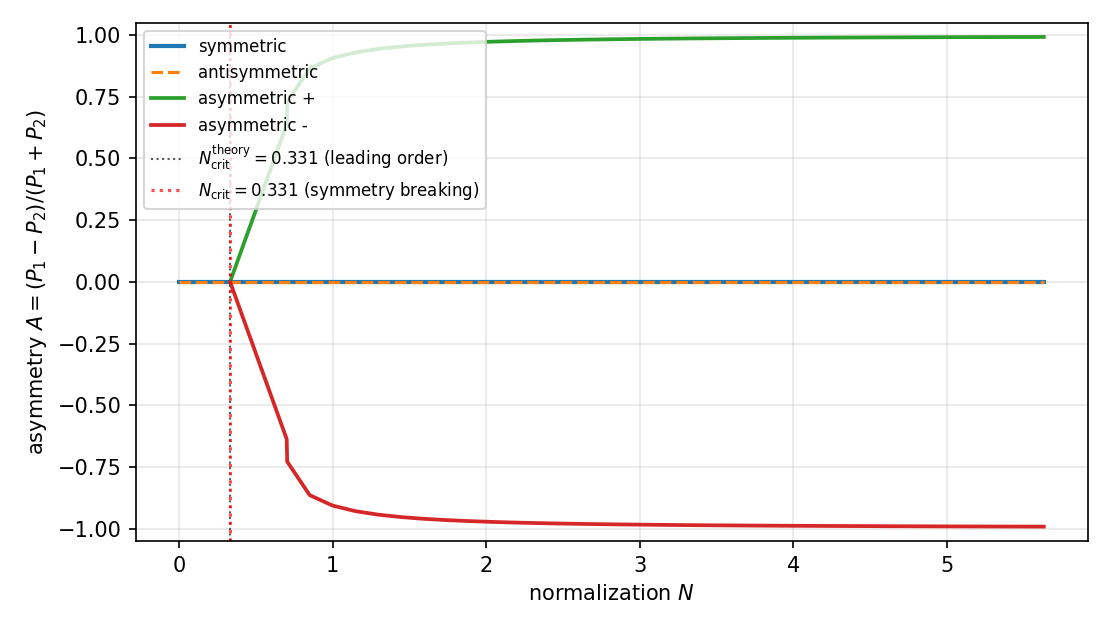}
    \caption{Asymmetry function plot for a dimer with gap ratio $g/R = 19$, distinguishing the two asymmetric branches shown in Figure \ref{fig:dimer-bifdiag-g19}.}
    \label{fig:dimer-bifdiag-asym-g19}
\end{figure}

\clearpage
\bibliographystyle{plain}
\bibliography{refs}

\end{document}